\documentclass[11pt]{article}
\usepackage{amsmath,amssymb,amsthm}
\usepackage{authblk}
\usepackage[utf8]{inputenc}
\usepackage{fontenc}
\usepackage{mathtools, nccmath, etoolbox}
\usepackage[a4paper]{geometry}
\usepackage[hidelinks]{hyperref} 
\usepackage{graphicx}
\usepackage{mathdots}
\usepackage{xcolor}
\usepackage{comment}
\theoremstyle{plain}
\newtheorem{theorem}{Theorem}
\newtheorem{proposition}[theorem]{Proposition}
\newtheorem{lemma}[theorem]{Lemma}
\newtheorem{corollary}[theorem]{Corollary}
\theoremstyle{definition}
\newtheorem{definition}[theorem]{Definition}
\newtheorem{example}{Example}
\theoremstyle{remark}
\newtheorem*{remark*}{Remark}
\newtheorem{remark}{Remark}
\newtheorem*{notation}{Notation}

\begin{document}

\title{A parallel iterative method for variational integration}

\author[1]{Sebasti\'an J. Ferraro\thanks{sferraro@uns.edu.ar}}
\author[2]{David Mart\'in de Diego\thanks{david.martin@icmat.es}}
\author[3]{Rodrigo Takuro Sato Martín de Almagro\thanks{rodrigo.t.sato@fau.de}}

\affil[1]{Instituto de Matem\'atica (INMABB), Departamento de Matem\'atica, Universidad Nacional del Sur (UNS) -- CONICET, Bah\'ia Blanca, Argentina \vspace*{0.2cm}}

\affil[2]{Instituto de Ciencias Matem\'aticas,  ICMAT (CSIC-UAM-UC3M-UCM)\authorcr
	Madrid, Spain \vspace*{0.2cm}}

\affil[3]{Institute of Applied Dynamics, Friedrich-Alexander-Universit\"at\authorcr Erlangen-N\"urnberg, Germany \vspace*{0.2cm}}
\date{}

\maketitle

\begin{abstract}

Discrete variational methods show excellent performance in numerical simulations of different mechanical systems. In this paper, we introduce an iterative procedure for the solution of discrete variational equations for boundary value problems. 
More concretely, we explore a parallelization strategy that leverages the capabilities of multicore CPUs and GPUs (graphics cards). We study this parallel method for higher-order Lagrangian systems, which appear in fully-actuated problems and beyond. The most important part of the paper is devoted to a precise study of different convergence conditions for these methods. We illustrate their excellent behavior in some interesting examples, namely Zermelo's navigation problem, a fuel-optimal navigation problem, interpolation problems or in a fuel optimization problem for a controlled 4-body problem in astrodynamics showing the potential of our method. 
\end{abstract}

\tableofcontents

\section{Introduction}

In this paper we propose a relaxation strategy to solving boundary value problems posed by variational integrators derived from discrete Hamilton’s principle \cite{MarsdenWest_ActaNum}. The algorithm can be implemented using a parallel computing approach, which can significantly improve its performance and simplify the way to find approximate solutions of the initial problem satisfying the boundary value conditions. Moreover, our techniques can be easily extended to more complex problems (see Example in Subsection \ref{example:interpolation}).

Parallelism is important since current hardware, namely multicore CPUs and most prominently GPUs (graphics cards), is especially designed for parallel computing. The cores in graphics cards are processing units that are simpler and slower than regular CPU cores. However, their number presently ranges from hundreds to thousands of cores per card. This allows for great performance gains via parallelization. The approach developed in our paper is scalable in the sense that once the algorithm for a given problem is written and tested on a GPU, additional or more powerful cards can be used to increase the number of cores and improve performance without changing the code.

Although our approach can be readily extended to general numerical methods for differential equations, in this paper, we restrict ourselves to numerical algorithms derived from the discrete Hamilton’s principle, called variational integrators \cite{MarsdenWest_ActaNum,hairer,blanes}.
It was in this setting where our strategy arose and it is this setting which allowed us to show convergence for a wide range of problems.

Discrete variational methods display an excellent long-term behaviour and preserve qualitative properties such as symmetries and constants of the motion, the manifold structure of the configuration space and geometric structures such as symplecticity or Poisson brackets. Since with our strategy we converge to solutions of these methods, we inherit these good properties.

Boundary value problems for general Lagrangian mechanical systems appear frequently in optimal control of mechanical systems and dynamic interpolation problems \cite{crouch-leite}. In \cite{McLachlanOffenBifurcationContinuous,McLachlanOffenBifurcationDiscrete,McLachlanOffenSymplecticBVP} the authors show that symplectic integrators and, in particular, variational integrators preserve very general types of bifurcations of Hamiltonian boundary value problems, something that standard methods generally do not achieve. Besides, variational integrators admit very natural extensions to other cases of interest such as systems on Lie groups, with external forces, holonomic and nonholonomic constraints or classical field theory, where it would be also possible to adapt the techniques developed in this paper.

The strategy is quite simple. Given a discretization of a continuous Lagragian $L\colon T Q \to \mathbb{R}$, $L_d \colon Q \times Q \rightarrow {\mathbb R}$, instead of solving its discrete Euler Lagrange equations (DEL)
\begin{equation}\label{eq:DEL}
	D_2 L_d(q_{k-1}, q_k) + D_1 L_d(q_k,q_{k+1})=0 \qquad \hbox{\bf (DEL equations)}
\end{equation}
exactly and at once, we produce a sequence $\{\bar q_{k}\}_{k=0}^N$ of points in $Q$ constructed iteratively using the following scheme: 
\begin{equation}\label{eq:pDEL.repeated}
    D_2 L_d(q_{k-1}, \bar q_k) + D_1 L_d(\bar q_k, q_{k+1}) = 0
\end{equation}
for each $k=1,\dots,N-1$, where we assume $\{q_{k}\}_{k=0}^N$ given and we solve in parallel for all $\bar q_{k}$, $1\leq k\leq N-1$. We can do this for higher-order Lagrangian theories, i.e. Lagrangians dependent on derivatives up to order $\gamma \geq 1$.

From a mathematical point of view one must ensure that the algorithm defined by Equation~\eqref{eq:pDEL.repeated} converges. In Section~\ref{sec:convergence} we show that if the Hessian of the matrix associated with the discrete action is positive-definite then our strategy converges locally to the solutions. To guarantee this, we prove two important results. The first one is Theorem~\ref{prp:discrete_positive_def} where we show that to prove convergence of the methods it is only necessary to check that the Hessian matrix of the discrete Lagrangian
\begin{equation*}
\mathrm{H}^{L_d}_k =
\left(
\begin{array}{cc}
D_{1 1} L_d(q_k,q_{k+1}) & D_{1 2} L_d(q_k,q_{k+1}) \\
D_{2 1} L_d(q_k,q_{k+1}) & D_{2 2} L_d(q_k,q_{k+1})
\end{array}
\right) = \left(
\begin{array}{cc}
     \mathcal{A}_k & \mathcal{C}_k \\
\mathcal{C}_k^\top & \mathcal{B}_k
\end{array}
\right),
\end{equation*}
is positive semi-definite and either $\mathcal{A}_k$ or $\mathcal{B}_k$ is positive-definite, for all $k=0,\dots,N-1$.

The second main result appears in Theorem~\ref{thm:Ld-from-L-defpos}, where we show when and how the properties of a continuous Lagrangian can automatically guarantee these definiteness properties for the Hessian of an approximate discrete Lagrangian.

Additionally, it should be noted that in the process of proving Theorem \ref{thm:Ld-from-L-defpos} we obtained some interesting results such as Proposition~\ref{prp:regularity}, where we check that the associated exact discrete Lagrangian for a positive-definite Lagrangian satisfies the conditions of convergence stated in Theorem~\ref{prp:discrete_positive_def}, or Proposition~\ref{prop-order}, where we rigorously derive the order of approximation of the Hessian matrix of a discrete Lagrangian that is obtained as a discretization of a continuous Lagrangian. We have added at the end of the paper an appendix with some technical results necessary to prove Theorem~\ref{thm:Ld-from-L-defpos}.
 
The power of the techniques developed in this paper is illustrated in some interesting problems related with navigation in Section \ref{sec-examples}. The first one is the classical Zermelo's navigation problem \cite{Zermelo}, a time optimal control problem where the trajectories are affected by a drift vector field (wind or water currents). We show that our methods quickly give us a set of local minimum time trajectories fixed initial and final conditions (see Section \ref{section-zermelo}). Modifications of this problem for fuel optimal navigation are also considered at the end of the paper (see also \cite{FeMaSaIFAC}) and a fuel minimization problem of a spacecraft moving under the gravitational force of three bodies (Sun-Earth and the Moon).

\section{Variational discrete equations}\label{sec:variationaldiscreteequations}

In this section, we will recall the theory behind the Hamilton's principle and variational integrators. Let $L_d\colon Q\times Q\rightarrow {\mathbb R}$ be a discrete Lagrangian derived from a discretization of a continuous Lagrangian $L\colon TQ\rightarrow {\mathbb R}$ (see \cite{MarsdenWest_ActaNum}). Moreover, we will also introduce discrete Lagrangians $L_d\colon T^{(\gamma-1)}Q\times T^{(\gamma-1)}Q\rightarrow {\mathbb R}$ derived from a $\gamma$-order Lagrangian system $L\colon T^{(\gamma)}Q\rightarrow {\mathbb R}$ following \cite{MR3562389}.
Here we denote by $T^{(\gamma)}Q$ the higher-order tangent bundle which consists of all equivalence classes of curves that agree up to their derivatives of order $\gamma$ (refer to \cite{Generalized-classical} for further details). Observe that we indistinctly denote $TQ\equiv T^{(1)}Q$.


\subsection{First-order systems}\label{section:first-order-system}
To simplify our exposition we will start with the standard case of first-order Lagrangians. 
The discrete Hamilton's principle states that for a discrete mechanical system on a configuration manifold $Q$ of dimension $n$, with a discrete Lagrangian $L_d\colon Q\times Q \to \mathbb{R}$, a sequence $\{q_k\}_{k=0}^{N}$ in $Q$ is a trajectory if and only if it satisfies the discrete Euler--Lagrange (DEL) equations~\eqref{eq:DEL} for $k=1,\dots,N-1$. These equations correspond to finding critical points of the discrete action $\sum_{k=0}^{N-1}L_d(q_k,q_{k+1})$ with fixed endpoints $q_0$ and $q_N$. 
Here, we will denote by $D_1L_d$ and $D_2L_d$ the derivatives of $L_d$ with respect to the first and second variables, respectively. 
See for instance \cite{MarsdenWest_ActaNum} and references therein.

One can start from a continuous Lagrangian $L\colon TQ\rightarrow {\mathbb R}$ and derive from it appropriate discrete Lagrangians in such a way that the DEL equations become a geometric integrator (variational integrator) for the continuous Euler--Lagrange equations
\begin{equation}\label{euler-lagrange}
\frac{d}{dt}\left( \frac{\partial L}{\partial \dot{q}}\right)- \frac{\partial L}{\partial q}=0
\end{equation}
where $(q, \dot{q})$ denotes local coordinates on $TQ$ induced by a system of coordinates $(q)$ on $Q$.

Starting from a continuous Lagrangian and somehow deriving a discrete Lagrangian, the DEL equations automatically provide a numerical integrator for the continuous Euler--Lagrange system, known as a variational integrator \cite{MarsdenWest_ActaNum}. Constructing numerical integrators from a discretization of Hamilton's principle instead of directly discretizing Equations~\eqref{euler-lagrange} implies that variational integrators are geometric integrators, i.e. they posses preservation of symplecticity, almost-preservation of energy and discrete momentum conservation \cite{MarsdenWest_ActaNum}.

Hence, given a regular Lagrangian function $L\colon TQ \to \mathbb{R}$, we define a discrete Lagrangian $L_d$ as an approximation of the exact discrete Lagrangian defined from the action of the continuous Lagrangian given by 
\[
L_{d}^{e}(q_0,q_1)=\int^h_0 L(q(t), \dot{q}(t))\, dt
\]
where $q\colon [0, h]\rightarrow Q$ is the unique solution of the Euler--Lagrange equations~\eqref{euler-lagrange} with initial and final conditions $q_0=q(0)$ and $q_1=q(h)$.  The discrete Lagrangian $L_d^e\colon TQ\rightarrow {\mathbb R}$ is known as the exact discrete Lagrangian and it  is well defined for a small enough time step $h$ and points $q_0$ and $q_1$ sufficiently close (see \cite{Hartman,MarsdenWest_ActaNum}).
\begin{definition}\label{order}
	Let $L_{d}\colon Q\times Q\rightarrow {\mathbb R}$ be a discrete Lagrangian. We say that
	$L_{d}$ is a discretization of order $r$ if there exist an
	open subset $U_{1}\subset TQ$ with compact closure and
	constants $C_1>0$, $h_1>0$ so that
	\begin{equation*}
	\lvert L_{d}(q(0),q(h))-L_{d}^{e}(q(0),q(h))\rvert\leq C_{1}h^{r+1}
	\end{equation*} for all solutions $q(t)$ of the second-order Euler--Lagrange equations with initial conditions $(q_0,\dot{q}_0)\in U_1$ and for all $h\leq h_1$.
\end{definition}

In \cite{MarsdenWest_ActaNum} and \cite{patrick} it is shown that if we have a discretization of order $r$ of the exact discrete Lagrangian then we obtain a numerical integrator for the Euler--Lagrange equations of a regular Lagrangian function $L\colon TQ\rightarrow {\mathbb R}$ with convergence order $r$. We take a discrete Lagrangian $L_{d}\colon Q\times Q\to {\mathbb R}$ 
as an approximation of $L_d^{e}$ and the order can be calculated by expanding the expressions for
$L_d(q(0),q(h))$ in a Taylor series in $h$ and comparing this to the same expansions for the exact Lagrangian.
If both series agree up to $r$ terms, then the discrete Lagrangian is of order $r$ (see \cite{MarsdenWest_ActaNum, Leok-shingel} and references therein).

\subsection{Higher-order systems}
Higher-order Lagrangian theories are systems where the Lagrangian depends on higher derivatives, i.e. velocities, accelerations and so on up to order $\gamma$. In \cite{MR3562389} a generalization of discrete variational calculus for higher-order Lagrangian systems $L\colon T^{(\gamma)} Q\rightarrow {\mathbb R}$ with $\gamma\geq 1$ was proposed. In particular, these results were proven to be useful for the discretization of fully actuated optimal control problems and interpolation problems. We recall the idea briefly. Consider a  Lagrangian $L\colon T^{(\gamma)}Q\to {\mathbb R}$. 
In the sequel, we will denote a point\footnote{Be sure to notice the distinction between the notations $q^{[\gamma]}$, which denotes a point, and $q^{(\gamma)}$, which denotes derivatives or adapted coordinates.} in $T^{(\gamma)} Q$ by $q^{[\gamma]}$, which in adapted local coordinates is $q^{[\gamma]}=(q^i \equiv q^{(0)\, i}, \dot{q}^i \equiv q^{(1)\, i}, ..., q^{(\gamma-1)\,i}, q^{(\gamma)\,i})$, $i = 1,...,n=\dim Q$

Observe that  when $\gamma=1$ we get $q^{(0)}=q\in Q$, and $q^{(1)}=\dot{q}$ which covers the case analyzed in Subsection \ref{section:first-order-system} regarding the standard case of discrete variational calculus. For $\gamma=2$ we get $q^{[2]}=(q, \dot{q}, \ddot{q})\in T^{(2)}Q$; as we will see in the examples section, this will be  useful in the optimal control of fully actuated systems and interpolation problems.

From now on, we will work with regular continuous Lagrangian systems, in accordance with the following definition. Note that we sometimes use the term ``continuous Lagrangian'' to distinguish it from the discrete Lagrangian. Also, $\gamma$-th order continuous and discrete Lagrangians are assumed to be at least of class $C^{2\gamma}$.
\begin{definition}\label{def:hessiano}
A (continuous) $\gamma$-th order Lagrangian, $L\colon T^{(\gamma)} Q \to \mathbb{R}$, with $\gamma\geq 1$ is said to be \textbf{regular} if the Hessian matrix
\begin{equation*}
\mathcal{W}(q^{[\gamma]}) = \left( \frac{\partial^2 L}{\partial q^{(\gamma)\,i} \partial q^{(\gamma)\,j}} \right)_{i,j = 1}^{n}
\end{equation*}
is regular.
\end{definition}

We say that a curve $q\colon [t_0, t_N]\to Q$ is {\bf critical} for the action 
\begin{equation} \label{ElasticSplines}
\mathcal{J} [q]:= \int_{t_0}^{t_N} L(q^{[\gamma]}(t))\; dt
\end{equation}
with $q^{[\gamma]}(t)=(q(t), \frac{d q}{dt}(t), \ldots, \frac{d q^{\gamma}}{dt^{\gamma}}(t))
$, where  $q^{[\gamma-1]}(0)$ and $q^{[\gamma-1]}(h)$ are fixed boundary conditions, if $\frac{\partial}{\partial \lambda}|_{\lambda=0}\mathcal{J} [q_\lambda]=0$ for all deformations $q_\lambda(t)$ of $q(t)$, $\lambda\in(-\epsilon,\epsilon)$, with fixed endpoints up to derivative order $\gamma-1$, that is
$q_0(t)=q(t)$ for $t\in [t_0, t_N]$, $q_\lambda^{(\alpha)}(u)=q^{(\alpha)}(u)$, $\alpha=0,\dots,\gamma-1$, $u\in\{t_0,t_N\}$.

A curve $q^{[\gamma]}$ is critical if and only if it is a solution of the Euler--Lagrange equations given by the system of $2\gamma$-order differential equations
\begin{equation}\label{gamma}
\sum_{\alpha=0}^{\gamma}
(-1)^{\alpha}\frac{d^{\alpha}}{dt^{\alpha}}\left(
\frac{\partial L}{\partial q^{(\alpha)}}\right)=0
\end{equation}

For the $\gamma$-th order case, a discrete Lagrangian  is given as a function 
$L_d\colon T^{(\gamma-1)}Q\times T^{(\gamma-1)}Q\rightarrow {\mathbb R}$ (see \cite{MR3562389}). The discrete action is a sum
\[
\sum_{k=0}^{N-1}L_d(q^{[\gamma-1]}_k,q^{[\gamma-1]}_{k+1})
\]
Observe that $T^{(\gamma-1)}Q$ is playing exactly the same role  as $Q$ in the DEL equations described in \eqref{eq:DEL}. The condition that a sequence $\{q_{k}^{[\gamma-1]}\}_{k=0}^N$ of points in $T^{(\gamma-1)}Q$ be critical for the discrete action, with fixed endpoints $q_{0}^{[\gamma-1]}$ and $q_{N}^{[\gamma-1]}$,
is equivalent to the equations
\begin{equation}\label{eq-deloc}
D_{2} L_d(q_{k-1}^{[\gamma-1]},q_k^{[\gamma-1]}) + D_{1} L_d(q_{k}^{[\gamma-1]},q_{k+1}^{[\gamma-1]})=0\qquad \hbox{\bf ($\boldsymbol{\gamma}$-th order DEL equations)}.
\end{equation}

\begin{definition}
\label{def:discrete_hessian}
A discrete  Lagrangian, $L_d\colon T^{(\gamma-1)} Q \times T^{(\gamma-1)} Q \to \mathbb{R}$, is said to be \textbf{regular} if its associated block matrix
\begin{equation*}
\mathcal{W}_d(q_0^{[\gamma-1]}, q_1^{[\gamma-1]}) 
= \left(
\begin{array}{cccc}
\frac{\partial^2 L_d}{\partial q_0 \partial q_1} & \frac{\partial^2 L_d}{\partial q_0 \partial \dot{q}_1} & \cdots & \frac{\partial^2 L_d}{\partial q_0 \partial q^{(\gamma-1)}_1}\\ 
\frac{\partial^2 L_d}{\partial \dot{q}_0 \partial q_1} & \frac{\partial^2 L_d}{\partial \dot{q}_0 \partial \dot{q}_1} & \cdots & \frac{\partial^2 L_d}{\partial \dot{q}_0 \partial q^{(\gamma-1)}_1}\\ 
\vdots & \vdots & \ddots & \vdots\\ 
\frac{\partial^2 L_d}{\partial q^{(\gamma-1)}_0 \partial q_1} & \frac{\partial^2 L_d}{\partial q^{(\gamma-1)}_0 \partial \dot{q}_1} & \cdots & \frac{\partial^2 L_d}{\partial q^{(\gamma-1)}_0 \partial q^{(\gamma-1)}_1}
\end{array}
\right)
\end{equation*}
is regular. 
\end{definition}

Starting from a continuous Lagrangian $L$, we define the exact discrete Lagrangian as
\begin{equation}\label {exact-gamma}
L_d^{e}(q^{[\gamma-1]}_0, q^{[\gamma-1]}_1)=
\int^h_0 L(q^{[\gamma]}(t))\; dt
\end{equation}
where $q\colon [0, h]\rightarrow Q$ is the unique  $C^{2 \gamma}$ solution curve of the Euler--Lagrange equations~\eqref{gamma} satisfying the boundary conditions 
$q^{[\gamma-1]}(0)=q^{[\gamma-1]}_0$ and $q^{[\gamma-1]}(h)=q^{[\gamma-1]}_1$ (see \cite{Agarwal}). This exact discrete Lagrangian is well-defined for $h$ small enough and   in a neighborhood $U_h$ of the diagonal of $T^{(\gamma-1)} Q \times T^{(\gamma-1)}Q$. We also know that $U_h$ degenerates into the diagonal for $h = 0$.

\begin{definition}\label{def:order-higher}
A discrete $\gamma$-th order Lagrangian $L_d$ is said to be an \textbf{approximation of order $\boldsymbol{r}$} (or \textbf{consistent to order $\boldsymbol{r}$}) with respect to a continuous Lagrangian $L$ if it agrees with the exact discrete Lagrangian of the latter up to order $r$, i.e., there exist an open set $U \subset T^{(2\gamma-1)} Q$ with compact closure and constants $C_U, h_U > 0$ such that
\begin{equation*}
\left\vert L_d(q^{[\gamma-1]}(0),q^{[\gamma-1]}(h)) - L_d^e(q^{[\gamma-1]}(0),q^{[\gamma-1]}(h))\right\vert \leq C_U h^{r+1}
\end{equation*}
for all solutions $q(t)$ of the Euler--Lagrange equations~\eqref{gamma} with initial values in $U$ and $h \leq h_U$.
\end{definition}

Under suitable regularity conditions, the DEL equations~\eqref{eq-deloc} can be used to find a trajectory sequentially. That is, one attempts to compute $q^{[\gamma-1]}_{k+1}$ using the previous points $q^{[\gamma-1]}_{k-1}$ and $q^{[\gamma-1]}_k$.
When solving boundary value problems with given initial and final conditions, some strategy should be adopted in order to arrive at the final desired condition. One such strategy is to apply a shooting method. For example, if $q^{[\gamma-1]}_0,q^{[\gamma-1]}_N\in T^{(\gamma-1)}Q$ and $N$ are given, one can try assigning some value to $q^{[\gamma-1]}_1$, run the sequential algorithm and compare the resulting $q^{[\gamma-1]}_N$ with the final condition; then adjust the value of $q^{[\gamma-1]}_1$ and repeat the process, until the final condition is met within a certain tolerance. However, for
optimal control problems this approach often fails to converge in practice, because of a high sensitivity of the final condition with respect to the starting guess, especially for $\gamma\geq 2$.  For this reason we propose a different, non-sequential strategy in the next section. 


\section{Parallel approach to the solution of the discrete equations}\label{sec:parallel_approach}

Consider the DEL equations \eqref{gamma}. Given $N\in \mathbb{N}$, $N\geq 2$, and given $q^{[\gamma-1]}_0,q^{[\gamma-1]}_N\in T^{(\gamma-1)}Q$, we want to find a sequence  $\{ q^{[\gamma-1]*}_k\}_{k=0}^{N}$, with $q^{[\gamma-1]*}_0=q^{[\gamma-1]}_0$, $q_N^{[\gamma-1]*}=q^{[\gamma-1]}_N$, that is a solution of \eqref{eq:DEL}. Our method starts with a sequence $\{q^{[\gamma-1]}_k\}$ chosen as the initial guess, with the only condition that $q^{[\gamma-1]}_0,q^{[\gamma-1]}_N$  are the points given, and produces a new sequence  $\{\bar q^{[\gamma-1]}_k\}$ with $\bar q^{[\gamma-1]}_0=q^{[\gamma-1]}_0$ and $\bar q^{[\gamma-1]}_N=q^{[\gamma-1]}_N$. In general, neither $\{q^{[\gamma-1]}_k\}$ nor $\{\bar q^{[\gamma-1]}_k\}$ will be a solution of \eqref{eq:DEL}, but by iterating this procedure we can approach a solution $\{q^{[\gamma-1]*}_k\}$, under certain assumptions to be specified in Section~\ref{sec:convergence}.

For each $k=1,\dots,N-1$, we find $\bar q^{[\gamma-1]}_k$ by solving a modified (``parallelized'') version of \eqref{eq:DEL}:
\begin{equation}\label{eq:pDEL}
  D_2L_d(q^{[\gamma-1]}_{k-1},\bar q^{[\gamma-1]}_k)+D_1L_d(\bar q^{[\gamma-1]}_k,q^{[\gamma-1]}_{k+1})=0  \qquad \hbox{\bf (Jacobi method)}.
\end{equation}
This means that for each triple $(q^{[\gamma-1]}_{k-1}, q^{[\gamma-1]}_k,q^{[\gamma-1]}_{k+1})$ of points in the sequence, the middle point moves to $\bar q_k$ so that the DEL equations hold for $(q^{[\gamma-1]}_{k-1}, \bar q^{[\gamma-1]}_k,q^{[\gamma-1]}_{k+1})$ (see Figure~\ref{fig:parallel_approach}). At the endpoints, we simply take $\bar q^{[\gamma-1]}_0=q^{[\gamma-1]}_0$ and $\bar q^{[\gamma-1]}_N=q^{[\gamma-1]}_N$. Computing $\bar q^{[\gamma-1]}_k$ for all $k$ completes one iteration, and the following one will use $\{\bar q^{[\gamma-1]}_k\}$ in place of $\{q^{[\gamma-1]}_k\}$. This approach is known as the nonlinear (block) Jacobi method \cite{Vrahatis2003, Axelsson_Iterative_Solution_Methods}, and we will discuss it in more detail in Section~\ref{sec:convergence}.

\begin{figure}
  \centering
  \includegraphics{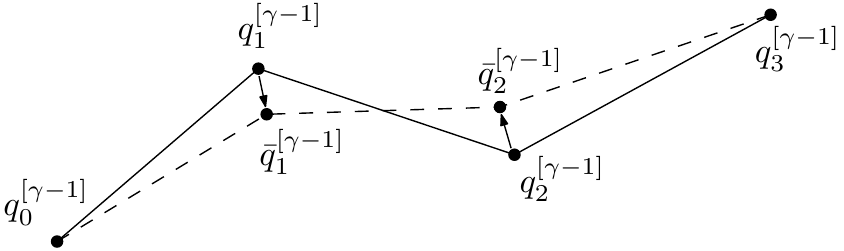}
  \caption{An iteration of the parallel method, for $N=3$.}
  \label{fig:parallel_approach}
\end{figure}

\begin{remark}
The solution $\bar q^{[\gamma-1]}_k$ of \eqref{eq:pDEL} can be found for each $k$ independently, using the data for the neighboring points from the latest iteration. Therefore, the procedure can be performed in a parallel fashion. The computed points $\bar q^{[\gamma-1]}_k$ remain unused until the next iteration.
\end{remark}

Assuming that $Q$ is a finite dimensional  vector space, a computationally less costly alternative is to replace \eqref{eq:pDEL} by a first order approximation. That is, instead of trying to solve the nonlinear system \eqref{eq:pDEL} exactly, we apply one step of the Newton--Raphson method to obtain a value for $\bar q^{[\gamma-1]}_k$, which clearly need not coincide with the exact solution of  \eqref{eq:pDEL}. This alternative update rule becomes
\begin{equation}\label{eq:linsysJN}
    \begin{split}
 \left(D_{22}L_d(q^{[\gamma-1]}_{k-1}, q^{[\gamma-1]}_{k})+D_{11}L_d(q^{[\gamma-1]}_{k}, q^{[\gamma-1]}_{k+1})\right)\cdot (\bar q^{[\gamma-1]}_k-
 q^{[\gamma-1]}_k)\\
   +D_2L_d(q^{[\gamma-1]}_{k-1}, q^{[\gamma-1]}_{k})+D_1L_d(q^{[\gamma-1]}_{k}, q^{[\gamma-1]}_{k+1})=0
  \end{split}  \qquad \hbox{\bf (Jacobi--Newton)}
\end{equation}
which means
\begin{equation}\label{eq:NewtonRaphson}
\bar q^{[\gamma-1]}_k=q^{[\gamma-1]}_k-\mathcal{D}_k^{-1}
\left(D_2L_d(q^{[\gamma-1]}_{k-1}, q^{[\gamma-1]}_{k})+D_1L_d(q^{[\gamma-1]}_{k}, q^{[\gamma-1]}_{k+1})\right),
\end{equation}
where
\begin{equation*}
\mathcal{D}_k = D_{22}L_d(q^{[\gamma-1]}_{k-1}, q^{[\gamma-1]}_{k})+D_{11}L_d(q^{[\gamma-1]}_{k}, q^{[\gamma-1]}_{k+1}).
\end{equation*}

Of course, it is necessary to assume that $\mathcal{D}_k$ is a regular matrix.
This procedure is known as the single-step Jacobi--Newton method (Section~\ref{sec:composite}). One could alternatively perform more Newton--Raphson substeps, iterating \eqref{eq:NewtonRaphson} two or more times within each Jacobi step.

These update rules are explicit and are also more suitable for parallel implementation, since the same expressions can be evaluated at all time steps simultaneously, with different values for the parameters $(q^{[\gamma-1]}_{k-1},q^{[\gamma-1]}_k,q^{[\gamma-1]}_{k+1})$. It requires solving $N-1$ linear systems of order $\gamma\dim Q$, a standard procedure for which there exist highly optimized implementations.
On the other hand, applying nonlinear solvers to \eqref{eq:pDEL} generally involves conditional statements which can cause the execution threads to diverge, that is, to execute different instructions. This can lead to a loss of performance in the parallel code.

For the single-step Jacobi--Newton method~\eqref{eq:linsysJN}, the matrix and vector coefficients involve second derivatives of $L_d$ and will typically have many common subexpressions; this can be taken into account to reduce the computational cost. It therefore makes sense to have a single procedure that computes them. The GPU hardware can apply this procedure to prepare the linear systems for all $k$ in parallel. Even though for $k=0$ and $k=N-1$ some of the derivatives are not needed (for example, $D_{11}L_d(q_0^{[\gamma-1]},q_1^{[\gamma-1]})$), the advantage of considering the common subexpressions makes this approach preferable.  In addition, the implementation of these computations can be simplified using libraries for automatic differentiation.

\section{Convergence}\label{sec:convergence}

In this section we will obtain sufficient conditions for the convergence of the iterative approach mentioned above. We will see that if the Hessian matrix of the discrete action is positive-definite at a solution, then both the Jacobi method and the Jacobi-Newton method converge locally to that solution. Afterwards, we will show that the positive-definiteness of the Hessian matrix follows from regularity conditions on the continuous Lagrangian and its discretization. 

\subsection{The Jacobi method}
The nonlinear Jacobi method is an iterative method for finding zeros of a nonlinear function $f=(f_1,\dots,f_n): \mathbb{R}^n \to \mathbb{R}^n$ (see \cite[p.\ 220]{OrtegaRheinboldt}). It is also called the method of simultaneous displacements. Starting from an initial guess $x^0=(x^0_1,\dots,x^0_n)\in \mathbb{R}^n$, the method generates a sequence $x^j\in\mathbb{R}^n$, $j=0,1,\dots$ that converges to a zero of $f$ under certain conditions. Namely, the $j$-th stage of the method consists in solving each scalar equation
\begin{equation}\label{eq:fi_for_Jacobi}
  f_k(x_1^j,\dots,x_{k-1}^j,x_k,x_{k+1}^j,\dots,x_n^j)=0, \quad k=1,\dots,n
\end{equation}
for $x_k$, independently for each $k$, and setting $x_k^{j+1}=x_k$, $k=1,\dots,n$, after all these equations are solved.

\begin{notation}
  In the literature on iterative methods it is usual to denote the successive approximations with a superscript as above. We adopt this notation for this section only, in order to discuss the convergence of the method; in the rest of the paper, we will use the previous notation $x_k \mapsto \bar x_k$ instead of $x_k^j \mapsto x_k ^{j+1}$ to denote the effect of a single iteration. 
\end{notation}

Other similar methods, such as Gauss-Seidel and successive overrelaxation (SOR), use the components of the new approximation $x^{j+1}$ as soon as they are available; however, the Jacobi method is better suited for parallel implementation. The article \cite{Vrahatis2003} gives the following result on the local convergence of the Jacobi method for finding critical points of a real-valued function. A matrix has the property $A^\pi$ mentioned in the theorem if it is block tridiagonal, possibly after conjugation by a permutation matrix, and the blocks on the diagonal are nonsingular \cite[p. 234]{Axelsson_Iterative_Solution_Methods}.

\begin{theorem}[\cite{Vrahatis2003}]\label{thm:Vrahatis}
  Let $F\colon \mathcal{D}\subset \mathbb{R}^n \to \mathbb{R}$ be twice continuously differentiable in an open neighborhood $\mathcal{S}_0\subset \mathcal{D}$ of a point $x^*\in \mathcal{D}$ for which $\nabla F(x^*)=0$, and suppose that the Hessian $\mathrm{H}(x^*)$ of $F$ is positive-definite with the property $A^\pi$. Then there exists an open ball $\mathcal{S}\subset \mathcal{S}_0$ centered at $x^*$ such that any sequence $\{x^j\}_{j=0}^\infty$, $x^0\in \mathcal{S}$, generated by the nonlinear Jacobi process converges to $x^*$.
\end{theorem}

Now consider a block partitioning of the equation $f(x)=0$, meaning that we regard
$\mathbb{R}^n\equiv \bigoplus_{k=1}^m \mathbb{R}^{d_k}$ and write $x=(\xi_1,\dots,\xi_m)$, $f=(\phi_1,\dots,\phi_m)$ in agreement with this splitting, that is, $d_k\in\mathbb{N}$, $\xi_k\in \mathbb{R}^{d_k}$ and $\phi_k\colon \mathbb{R}^n \to \mathbb{R}^{d_k}$ for all $k$. Note that the domain and codomain of $f$ are partitioned in the same way. The $j$-th stage of the \emph{block} nonlinear Jacobi method consists in solving $\phi_k(\xi_1^j,\dots,\xi_{k-1}^j,\xi_k,\xi_{k+1}^j,\dots,\xi_m^j)=0$ for $\xi_k$, for each $k=1,\dots,m$, and setting $\xi_k^{j+1}=\xi_k$. The approach we propose in equation \eqref{eq:pDEL} is then a block nonlinear Jacobi method, with blocks of equal size $\gamma \dim Q$.
\begin{remark}
Theorem \ref{thm:Vrahatis} is stated in \cite{Vrahatis2003} for the non block partitioned case (i.e., for scalar equations and unknowns), although the same argument in their proof is valid for the \emph{block} Jacobi method.
\end{remark}

Equations~\eqref{eq-deloc} are of the form \eqref{eq-deloc} are $\nabla F=0$ where $F = \sum_{k=0}^{N-1} L_d(q^{[\gamma-1]}_k,q^{[\gamma-1]}_{k+1})$ as a function of $x=q_d^{[\gamma-1]}=(q^{[\gamma-1]}_1,\dots,q^{[\gamma-1]}_{N-1})$, and $q^{[\gamma-1]}_0$ and $q^{[\gamma-1]}_N$ are fixed. Since $L_d$ is $C^2$, then the Hessian of $F$ is symmetric and has the block tridiagonal form
  \begin{equation}\label{eq:blocktridiagonal}
\mathrm{H}(q_d^{[\gamma-1]})=
    \begin{pmatrix}
      \mathcal{D}_1 & \mathcal{C}_1 &  & \\
      \mathcal{C}_1^{\top} & \mathcal{D}_2 & \mathcal{C}_2 & & \\
          & \ddots & \ddots & \ddots &\\
          &        & \mathcal{C}_{N-3}^{\top} & \mathcal{D}_{N-2} & \mathcal{C}_{N-2}\\
          &        &         & \mathcal{C}_{N-2}^{\top} & \mathcal{D}_{N-1}
    \end{pmatrix}
  \end{equation}
  where
  \begin{align*}
\mathcal{D}_k&=D_{22}L_d(q^{[\gamma-1]}_{k-1},q^{[\gamma-1]}_k)+D_{11}L_d(q^{[\gamma-1]}_k,q^{[\gamma-1]}_{k+1}), \quad k=1,\dots,N-1,\\
\mathcal{C}_k&=D_{12}L_d(q^{[\gamma-1]}_k,q^{[\gamma-1]}_{k+1}), \quad k=1,\dots,N-2.
  \end{align*}
  The following is a direct consequence of Theorem~\ref{thm:Vrahatis}.
\begin{proposition}\label{prop:Jacobi_convergence}
  Let $q_d^{[\gamma-1]*}=(q^{[\gamma-1]*}_1,\dots,q^{[\gamma-1]*}_{N-1})$ be a solution of the DEL equations for fixed $q^{[\gamma-1]}_0$ and $q^{[\gamma-1]}_N$. If the Hessian of the discrete action, $\mathrm{H}(q_d^{[\gamma-1]*})$, is positive-definite, then the block Jacobi method converges locally to $q_d^{[\gamma-1]*}$. 
\end{proposition}
The hypothesis that $\mathrm{H}(q_d^{[\gamma-1]*})$ is positive-definite implies that $\mathcal{D}_k(q_d^{[\gamma-1]*})$ is regular for $k=1,\dots,N-1$ (see the proof of Proposition \ref{prop:Hpd_implies_JN_conv}), which is required by condition $A^\pi$.

\begin{example}
A basic situation in mechanics ($\gamma=1$) occurs when the discrete Lagrangian can be locally written as
\begin{equation*}
    L_d(q_0,q_1)= \frac{1}{2h}(q_1-q_0)^{\top} M (q_1-q_0) - hV(q_0,q_1)
\end{equation*}
  where $M$ is a constant, symmetric, positive-definite (p.d.)\ matrix, and $V$ is linear in $(q_0,q_1)$. We have $\mathcal{D}_k=2M/h$ and $\mathcal{C}_k=-M/h$ for all $k$, so $\mathrm{H}(q_d^{[0]})=\frac{1}{h} K \otimes M$ for all $q_d^{[0]}$, where $\otimes$ denotes the Kronecker product and $A$ is the $(N-1)\times(N-1)$ tridiagonal matrix
  \[
    K=
    \begin{pmatrix}
      2 & -1 &  & \\
      -1 & 2 & -1 & & \\
          & \ddots & \ddots & \ddots &\\
          &        & -1 & 2 & -1\\
          &        &         & -1 & 2
    \end{pmatrix}.
  \]
  The matrix $K$ is p.d.\ since if we write $x=(x_1,\dots,x_{N-1})$ then
  \[
    x^{\top}Kx=x_1^2+\sum_{k=1}^{N-2}(x_k-x_{k+1})^2+x_{N-1}^2>0\quad \text{for }x\neq 0.
 \]
Then $\mathrm{H}(q_d^{[0]})$ is p.d., being the Kronecker product of two p.d.\ matrices. We will extend this example to a more general case in Section~\ref{sec:Hessianpd}.
\end{example}

\subsection{The Jacobi--Newton composite method}\label{sec:composite}

When trying to solve each scalar equation \eqref{eq:fi_for_Jacobi} in the Jacobi method, one can apply for instance $m$ Newton--Raphson steps, thus obtaining the composite method known as $m$-step Jacobi--Newton.

Alternatively, a step of the Newton--Raphson method applied to the full nonlinear system $f(x)=0$, $x\in \mathbb{R}^n$, consists in solving
\begin{equation}\label{eq:fullNewton}
  Df(x^j)(x^{j+1}-x^j)=-f(x^j)
\end{equation}
for $x^{j+1}$, starting from an initial guess $x^0$. When $n$ is large, $Df(x^j)^{-1}$ is not readily available and this linear system must be solved using an iterative method; for instance, if $m$ steps of the (linear) Jacobi method are applied, one obtains what is called the $m$-step Newton--Jacobi method. The one-step Newton--Jacobi and one-step Jacobi--Newton methods actually coincide for the scalar (non-partitioned) case \cite[p.\ 221]{OrtegaRheinboldt}. Furthermore, they also coincide for the block partitioned versions of these methods, since the same arguments in the proof in \cite{OrtegaRheinboldt} hold in that case.
Write
\[
  Df(x)=\mathcal{D}(x)+\mathcal{C}(x),\quad J(x)=-\mathcal{D}(x)^{-1}\mathcal{C}(x),
\]
where $\mathcal{D}(x)$ is the block diagonal part of $Df(x)$. The block one-step Newton--Jacobi method converges locally to a solution $x^*$ of the nonlinear system if $f$ is smooth, $\mathcal{D}(x^*)$ is nonsingular and $\rho(J(x^*))<1$, where $\rho(X)$ denotes the spectral radius of $X$ \cite[p.\ 321, p.\ 332]{OrtegaRheinboldt}.

The alternative update rule \eqref{eq:NewtonRaphson} we propose is one Newton step for a block Jacobi--Newton method. In our case, $Df(x)$ is the Hessian matrix \eqref{eq:blocktridiagonal}, and the matrix $\mathcal{D}(x)$ is $\operatorname{diag}(\mathcal{D}_1,\dots,\mathcal{D}_{N-1})$. Taking into account the equivalence of the methods, the convergence result just mentioned becomes the following.
\begin{lemma}\label{prop:rhoJ}
  Let $q_d^{[\gamma-1]*}=(q^{[\gamma-1]*}_1,\dots,q^{[\gamma-1]*}_{N-1})$ be a solution of the DEL equations for fixed $q^{[\gamma-1]}_0$ and $q^{[\gamma-1]}_N$. If the blocks $\mathcal{D}_1,\dots,\mathcal{D}_N$ on the diagonal of $\mathrm{H}(q_d^{[\gamma-1]*})$ are regular and $\rho(J(q_d^{[\gamma-1]*}))<1$, then the 1-step block Jacobi--Newton method converges locally to $q_d^{[\gamma-1]*}$.
\end{lemma}

\begin{proposition}\label{prop:Hpd_implies_JN_conv}
  Let $q_d^{[\gamma-1]*}=(q^{[\gamma-1]*}_1,\dots,q^{[\gamma-1]*}_{N-1})$ be a solution of the DEL equations, for fixed $q^{[\gamma-1]}_0$ and $q^{[\gamma-1]}_N$. If the Hessian $\mathrm{H}(q_d^{[\gamma-1]*})$ is p.d., then the 1-step block Jacobi--Newton method converges locally to $q_d^{[\gamma-1]*}$.
\end{proposition}
\begin{proof}
  Denote by $\mathcal{D}_1,\dots,\mathcal{D}_{N-1}$ the diagonal blocks that form $\mathcal{D}(q_d^{[\gamma-1]*})$, the block diagonal part of $\mathrm{H}(q_d^{[\gamma-1]*})$. Since $\mathrm{H}(q_d^{[\gamma-1]*})$ is symmetric and p.d., all its leading principal minors are positive, by Sylvester's criterion. This implies that $\mathcal{D}_1$ is also p.d.. Using a suitable permutation matrix $P$, one can move any $\mathcal{D}_k$ to the upper left position. Since $P^{\top} \mathrm{H}(q_d^{[\gamma-1]*})P$ is also symmetric and  p.d., every $\mathcal{D}_k$ is p.d..
Therefore $\mathcal{D}(q_d^{[\gamma-1]*})$ is p.d.. Theorem 6.38 in \cite{Axelsson_Iterative_Solution_Methods} shows that if $\mathrm{H}(q_d^{[\gamma-1]*})$ is p.d., block tridiagonal and $\mathcal{D}(q_d^{[\gamma-1]*})$ is p.d., then 
$\rho(J(q_d^{[\gamma-1]*}))<1$. The convergence result now follows from Lemma~\ref{prop:rhoJ}.
\end{proof}

\begin{remark}\label{rem:Hpd-implies-convergence}
If the Hessian is p.d., then both the block Jacobi method and the block one-step Jacobi--Newton method converge locally (Propositions \ref{prop:Jacobi_convergence} and \ref{prop:Hpd_implies_JN_conv}).
\end{remark}

\begin{remark}
Taking more than one Newton step in the Jacobi--Newton method does not generally enhance the rate of convergence \cite[p.\ 327]{OrtegaRheinboldt}.
\end{remark}

\subsection{Single-step sufficient conditions for the convergence of the methods}\label{sec:Hessianpd}

Even though the positive-definiteness of the Hessian for the full system of equations implies the local convergence of both the block Jacobi and the block (one-step) Jacobi--Newton methods, it can be a difficult condition to check in practice, as it is a matrix of order $\gamma(N-1)\dim Q$. In this section we obtain conditions ensuring that this Hessian is p.d. that rely solely on the Hessian matrix of $L_d\colon T^{(\gamma-1)}Q\times T^{(\gamma-1)}Q\to \mathbb{R}$, which is a matrix of order $2\gamma\dim Q$. We also show that if the continuous Lagrangian $L\colon T^{(\gamma)}Q\to \mathbb{R}$ has a positive-definite Hessian and $L_d$ is an approximation of order $2\gamma-1$, then the latter conditions
hold.

\subsection{The purely discrete case}
Let us rewrite equation \eqref{eq:blocktridiagonal} as
\begin{equation}
\mathrm{H}(q_d^{[\gamma-1]}) = \left(
\begin{array}{ccccc}
\mathcal{B}_0 + \mathcal{A}_1 &                 \mathcal{C}_1 &                        &                                       &                                      \\
           \mathcal{C}_1^\top & \mathcal{B}_1 + \mathcal{A}_2 &          \mathcal{C}_2 &                                       &                                      \\
                              &                        \ddots &                 \ddots &                                \ddots &                                      \\
                              &                               & \mathcal{C}_{N-3}^\top & \mathcal{B}_{N-3} + \mathcal{A}_{N-2} &                     \mathcal{C}_{N-2}\\
                              &                               &                        &                \mathcal{C}_{N-2}^\top & \mathcal{B}_{N-2} + \mathcal{A}_{N-1}
\end{array}
\right)
\label{eq:hessian_decomp_2}
\end{equation}
where
\begin{align*}
\mathcal{A}_k &= D_{1 1} L_d(q_k^{[\gamma-1]},q_{k+1}^{[\gamma-1]}), \quad k = 1, \dots, N - 1,\\
\mathcal{B}_k &= D_{2 2} L_d(q_k^{[\gamma-1]},q_{k+1}^{[\gamma-1]}), \quad k = 0, \dots, N - 2,\\
\mathcal{C}_k&=D_{12}L_d(q^{[\gamma-1]}_k,q^{[\gamma-1]}_{k+1}), \quad k=1,\dots,N-2,
\end{align*}
and $q^{[\gamma-1]}_0$, $q^{[\gamma-1]}_N$ are fixed. Observe that ${\mathcal D}_{k}=\mathcal{B}_{k-1} + \mathcal{A}_k$, $k=1, \ldots, N-1$. 
The following result is a straightforward verification.
\begin{lemma}
\label{lem:quadratic_form_decomp}
Assume $\mathcal{Q}\colon \mathbb{R}^n \times \mathbb{R}^n\to \mathbb{R}$ is a quadratic form whose associated matrix has the same block structure as the rhs of Eq.~\eqref{eq:hessian_decomp_2}. Let $v = (v_1, v_2, ..., v_{N-1}) \in \mathbb{R}^n$, with $v_i$, $i = 1, ..., N-1$, of the right dimension corresponding to the block structure decomposition of $\mathcal{Q}$. Then,
\begin{equation*}
\mathcal{Q}(v) = v_1^{\top} \mathcal{B}_0 v_1 + \sum_{k = 1}^{N-2} \mathcal{Q}_k((v_k,v_{k+1})) + v_{N-1}^{\top} \mathcal{A}_{N-1} v_{N-1},
\end{equation*}
with
\begin{equation*}
\mathcal{Q}_k((v_k,v_{k+1})) = \left(
\begin{array}{cc}
v_k^{\top} & v_{k+1}^{\top}
\end{array}\right)
\left(
\begin{array}{cc}
\mathcal{A}_k &       \mathcal{C}_k \\
\mathcal{C}_{k}^\top & \mathcal{B}_{k}
\end{array}
\right)
\left(
\begin{array}{c}
    v_k \\
v_{k+1}
\end{array}
\right)\,.
\end{equation*}
\end{lemma}

\begin{remark}
If we define $v_0, v_N, \mathcal{A}_0, \mathcal{C}_0, \mathcal{C}_{N-1}, \mathcal{B}_{N-1}$ of the appropriate dimensions and set $v_0 = 0$ and $v_{N} = 0$, then with the same definitions we may write
\begin{equation*}
\mathcal{Q}(v) = \sum_{k = 0}^{N-1} \mathcal{Q}_k((v_k,v_{k+1}))\,.
\end{equation*}
\end{remark}
\begin{notation}
The expression $(v_k,v_{k+1})$ introduced above represents the concatenation or direct sum of $v_k$ and $v_{k+1}$.
\end{notation}

\begin{theorem}
\label{prp:discrete_positive_def}
Let $(q_0^{[\gamma-1]}, q_1^{[\gamma-1]}, \dots, q_N^{[\gamma-1]})$ be an arbitrary sequence in $T^{(\gamma-1)}Q$, and denote $q_d^{[\gamma-1]} = (q_1^{[\gamma-1]}, \dots, q_{N-1}^{[\gamma-1]})$. Consider a discrete $\gamma$-th order Lagrangian $L_d\colon T^{(\gamma-1)} Q \times T^{(\gamma-1)} Q \to \mathbb{R}$ of class $C^2$. If the Hessian of the discrete Lagrangian 
\begin{equation*}
\mathrm{H}^{L_d}_k (q_k^{[\gamma-1]},q_{k+1}^{[\gamma-1]}) =
\left(
\begin{array}{cc}
D_{1 1} L_d(q_k^{[\gamma-1]},q_{k+1}^{[\gamma-1]}) & D_{1 2} L_d(q_k^{[\gamma-1]},q_{k+1}^{[\gamma-1]}) \\
D_{2 1} L_d(q_k^{[\gamma-1]},q_{k+1}^{[\gamma-1]}) & D_{2 2} L_d(q_k^{[\gamma-1]},q_{k+1}^{[\gamma-1]})
\end{array}
\right) = \left(
\begin{array}{cc}
\mathcal{A}_k &       \mathcal{C}_k \\
\mathcal{C}_{k}^\top & \mathcal{B}_{k}
\end{array}
\right),
\end{equation*}
is positive semi-definite for all $k=0,\dots,N-1$, then $\mathrm{H}(q_d^{[\gamma-1]})$ is positive semi-definite. If, in addition, $\mathcal{A}_k$ is positive-definite for all $k$ or $\mathcal{B}_k$ is positive-definite for all $k$, then so is $\mathrm{H}(q_d^{[\gamma-1]})$, and the block Jacobi and block 1-step Jacobi--Newton methods converge locally.
\end{theorem}

\begin{proof} Write
\begin{equation*}
\mathcal{Q}_k((v_k,v_{k+1})) := \mathrm{H}^{L_d}_k(v_k,v_{k+1}), \quad k = 0,...,N-1,
\end{equation*}
where $v_k\in\mathbb{R}^{\gamma\dim Q}$, $k=0,\dots,N$.

Let $v=(v_1,\dots,v_{N-1})\in \mathbb{R}^{\gamma(N-1)\dim Q}$ be nonzero, and define $v_0=v_N=0\in \mathbb{R}^{\gamma\dim Q}$. Since by hypothesis $\mathrm{H}^{L_d}_k$ is positive semi-definite for all $k$, then
\begin{align*}
\mathrm{H}(q_d^{[\gamma-1]})(v,v) &= \sum_{k = 0}^{N-1} \mathcal{Q}_k((v_k,v_{k+1})) \\ &=\mathcal{Q}_0((0,v_1)) + \sum_{k = 1}^{N-2} \mathcal{Q}_k((v_k,v_{k+1})) + \mathcal{Q}_{N-1}((v_{N-1},0)) \geq 0\,,
\end{align*}
proving the first claim.

For the second claim suppose that, in addition, $\mathcal{B}_k$ is positive-definite for all $k=0,\dots N-1$, and let $k_0$ be the first index such that $v_{k_0}\neq 0$. Then the $k_0$-th term of the above sum is \[\mathcal{Q}_{k_0-1}((0,v_{k_0}))
=v_{k_0}^\top \mathcal{B}_{k_0-1}v_{k_0}>0.
\]
As all other terms are nonnegative, $\mathrm{H}(q_d^{[\gamma-1]})(v,v)>0$.
Similarly, if $\mathcal{A}_k$ is positive-definite for all $k$ and $k_0$ is the last index such that $v_{k_0}\neq 0$, we obtain a nonzero term
\[
\mathcal{Q}_{k_0}((v_{k_0},0))
=v_{k_0}^\top \mathcal{A}_{k_0}v_{k_0}>0.
\]
In either case, $\mathrm{H}(q_d^{[\gamma-1]})$ is positive-definite, and the convergence of the methods follows from Remark~\ref{rem:Hpd-implies-convergence}.
\end{proof}

\begin{example}
\label{exp:trapezoidal}
Consider the family of discrete second-order Lagrangians
$L^{\alpha}_d\colon T{\mathbb R}^n \times T {\mathbb R}^n \rightarrow {\mathbb R}$
with parameter $\alpha \in \mathbb{R}$ based on the trapezoidal rule, given by
\begin{equation*}
L_d^{\alpha}(q_0,v_0,q_1,v_1) = \frac{h}{2} \left[ L\left(q_0,v_0,a_0^{\alpha}\right) + L\left(q_1,v_1,a_1^{\alpha}\right)\right]
\end{equation*}
where
\begin{align*}
a_0^{\alpha} := \frac{\left[(1 - 3 \alpha) v_1 - (1 + 3 \alpha) v_0\right] \,h + 6 \alpha (q_1 - q_0)}{h^2}\,,\\
a_1^{\alpha} := \frac{\left[(1 + 3 \alpha) v_1 - (1 - 3 \alpha) v_0\right] \,h - 6 \alpha (q_1 - q_0)}{h^2}\,,
\end{align*}
with $L\colon T^{(2)}Q\rightarrow {\mathbb R}$ a continuous Lagrangian with positive-definite Hessian. 

It can be checked that this family of discrete Lagrangians is regular for $\alpha \neq 0$ and it is a discrete approximation of order $2$ for $L$ in the sense of Definition \ref{def:order-higher}. In this case we have that the Hessian  of the discrete Lagrangian $L_d^{\alpha}$ are
\begin{equation}\label{eq:trapezoidal}
\mathrm{H}_k^{L^{\alpha}_d}
=
\left(
\begin{array}{rrrr}
 \frac{4 (3 \alpha)^2}{h^3} &  \frac{2 (3 \alpha)^2}{h^2} & -\frac{4 (3 \alpha)^2}{h^3} &  \frac{2 (3 \alpha)^2}{h^2}\\[3pt]
 \frac{2 (3 \alpha)^2}{h^2} &   \frac{(3\alpha)^2 + 1}{h} & -\frac{2 (3 \alpha)^2}{h^2} &   \frac{(3\alpha)^2 - 1}{h}\\[3pt]
-\frac{4 (3 \alpha)^2}{h^3} & -\frac{2 (3 \alpha)^2}{h^2} &  \frac{4 (3 \alpha)^2}{h^3} & -\frac{2 (3 \alpha)^2}{h^2}\\[3pt]
 \frac{2 (3 \alpha)^2}{h^2} &   \frac{(3\alpha)^2 - 1}{h} & -\frac{2 (3 \alpha)^2}{h^2} &   \frac{(3\alpha)^2 + 1}{h}
\end{array}
\right)\otimes \mathcal{W}(q(\xi_k), \dot{q}(\xi_k), \ddot{q}(\xi_k))+ \mathrm{h.o.t.}
\end{equation}
where we $\xi_k \in [t_k, t_{k+1}]$ with $t_k = t_0 + k h$ and $q(t)$ denotes the continuous solution of the Euler--Lagrange for $L$. 
With the notation ``h.o.t."  we mean the matrix corresponding to higher order terms of each element of matrix $\mathrm{H}_k^{L^{\alpha}_d}$. 
Therefore, in order to check if we satisfy the hypotheses of Theorem~\ref{prp:discrete_positive_def}, for $h$ sufficiently small we only need to analyse if the first matrix of Equation~\eqref{eq:trapezoidal} fulfills the conditions of the theorem. However, observe that this matrix is positive semi-definite since 
every principal minor is $\geq 0$. Notice that contrary to the positive-definite case one needs to compute all principal minors and not only the leading ones. Moreover, since the first $2 \times 2$ block is positive-definite, the block Jacobi and block 1-step Jacobi--Newton methods converge for this discrete Lagrangian. 
\end{example}

\subsection{Discrete Lagrangians derived from a continuous one}
Now consider a regular Lagrangian of order $\gamma$.
From \cite{MR3562389} we know that for sufficiently small $h$ there exists a $C^{2 \gamma}$ unique solution curve of Hamilton's principle, defined by a point in a neighborhood $U_h$ of the diagonal of $T^{(\gamma-1)} Q \times T^{(\gamma-1)}Q$. We also know that $U_h$ degenerates into the diagonal for $h = 0$.

Therefore, away from $h = 0$, a $C^{2 \gamma}$ curve $q\colon [0,h] \to Q$ assumed to be a solution of Hamilton's 
principle may be parametrized by a point in $U_h$, i.e. a pair of sufficiently close points $q^{[\gamma-1]}_0$, $q^{[\gamma-1]}_1 \in T^{(\gamma-1)} Q$ that we may consider as boundary values of the lift of $q$ to
$T^{(\gamma-1)} Q$ (see \cite{Agarwal, Hartman,MR3562389} for more details). These results imply the existence of an exact discrete Lagrangian.

\begin{lemma}
\label{lem:IVP_BVP}
Let $q\colon [0, h] \to U \subset Q$, with $h > 0$, be a solution curve for a regular Lagrangian $L\colon T^{(\gamma)}Q\to \mathbb{R}$, and $(\varphi, U)$ a local chart. Assume that the lift of $q$ to $T^{(2 \gamma - 1)} Q$ in the corresponding adapted coordinate system takes the form $q^{[2 \gamma - 1]}(t) = (q^i(t),\dot{q}^i(t), \ddot{q}^{\,i}(t),..., q^{(2 \gamma - 1)\,i}(t))$, $i = 1,...,n=\dim Q$, and define $q^{(\alpha)\,i}_0 := q^{(\alpha)\,i}(0)$, $q^{(\alpha)\,i}_1 := q^{(\alpha)\,i}(h)$ for $\alpha = 0,..., 2 \gamma - 1$. Then, for $\alpha, \beta = 0,...,\gamma - 1$, $u, v = 0, 1$ and $i,j=1,...,n$,
\begin{equation*}
\frac{\partial q^{(\alpha)\,i}_u}{\partial q^{(\beta)\,j}_v} = \delta_{\alpha \beta} \delta_{u v}\delta_{ij}
\end{equation*}
and for $\alpha = \gamma, ..., 2 \gamma - 1$, $\beta = 0,...,\gamma - 1$,
\begin{equation*}
\frac{\partial q^{(\alpha)\,i}_u}{\partial q^{(\beta)\,j}_v} = \delta_{ij}\mathcal{O}( h^{-\alpha + \beta} )\,.\
\end{equation*}
\end{lemma}

\begin{proof}
The first statement is trivial, since it involves independent boundary conditions. The second statement can be shown as follows. By Taylor's theorem we have that for $\alpha = 0, ..., 2 \gamma - 1$,
\begin{equation*}
q^{(\alpha)\,i}(h) = \sum_{\beta = 0}^{2 \gamma - \alpha - 1} h^{\beta} \frac{q^{(\alpha + \beta) \, i}(0)}{\beta!} + \mathcal{O}(h^{2 \gamma - \alpha})\,,
\end{equation*}
and similarly
\begin{equation*}
q^{(\alpha)\,i}(0) = \sum_{\beta = 0}^{2 \gamma - \alpha - 1} (-h)^{\beta} \frac{q^{(\alpha + \beta) \, i}(h)}{\beta!} + \mathcal{O}(h^{2 \gamma - \alpha}).
\end{equation*}

Defining $x_u = (q^{(0)\,i}_u, ..., q^{(\gamma-1)\,i}_u)^{\top}$, $y_u = (q^{(\gamma)\,i}_u, ..., q^{(2\gamma-1)\,i}_u)^{\top}$, for $u = 0,1$, and $z = (h^{2 \gamma}, ..., h^{\gamma + 1})^{\top}$, we can rewrite the first $\gamma$ equations of each in matrix form as
\begin{align*}
x_1 &= \mathbf{A}(h) x_0 + \mathbf{B}(h) \mathbf{E} y_0 + \mathcal{O}(z)\\
x_0 &= \mathbf{A}(-h) x_1 + \mathbf{B}(-h) \mathbf{E} y_1 + \mathcal{O}(z)\,,
\end{align*}
where we have used the matrices defined in the Appendix. By Corollary \ref{cor:B_matrix}, we know that both $\mathbf{B}(h)$ and $\mathbf{B}(-h)$ are invertible. Taking into account that $\mathbf{E}^{-1} = \mathbf{E}$, we immediately find that
\begin{equation}
\begin{aligned}\label{eq:y01_intermsof_x01}
y_0 &= \mathbf{E} \mathbf{B}^{-1}(h) ( x_1 - \mathbf{A}(h) x_0) + \mathcal{O}(z h^{-\gamma})\,,\\
y_1 &= \mathbf{E} \mathbf{B}^{-1}(-h) (- \mathbf{A}(-h) x_1 + x_0) + \mathcal{O}(z h^{-\gamma})\,,
\end{aligned}
\end{equation}
and a simple computation leads us to the desired result.
\end{proof}

\begin{remark}
Following \cite{MR3562389} there exist local diffeomorphisms $\Psi_0, \Psi_1\colon T^{(\gamma-1)}Q \times T^{(\gamma-1)} Q \to T^{(2 \gamma - 1)} Q$, such that $\Psi_u(q^{[\gamma-1]}_0,q^{[\gamma-1]}_1) = q^{[2 \gamma - 1]}_u$ with $u = 0, 1$, defined on a tubular neighborhood of the diagonal that are naturally induced by the jet-like structure of these spaces. The lemma above computes their derivatives. The specific form of these diffeomorphisms is controlled by the vector fields that generate them, and in particular, the Euler--Lagrange equations in the case we are interested.
\end{remark}

\begin{proposition}
\label{prp:regularity}
Let $L\colon T^{(\gamma)} Q \to \mathbb{R}$ be a $C^{2\gamma}$ regular $\gamma$-th order Lagrangian with positive-definite Hessian matrix. Then there exists a deleted neighborhood of $0$, $U_0^* = U_0 \setminus \left\lbrace 0 \right\rbrace$, with $U_0 \subseteq \mathbb{R}$, for which the exact discrete Lagrangian, $L_d^{h,e} \equiv L_d^e\colon T^{(\gamma-1)} Q \times T^{(\gamma-1)} Q \to \mathbb{R}$, with $h \in U_0^*$, is regular. Moreover, if  $h\in U^*_{0,+} := U_0^* \cap \mathbb{R}_{+}$, $L_d^e$ satisfies all the hypotheses in Theorem~\ref{prp:discrete_positive_def}.
\end{proposition}

\begin{proof}
We know that, by its definition, the partial derivatives of the exact discrete Lagrangian coincide with the Jacobi--Ostrogradski momenta \cite{Generalized-classical}:
\begin{align*}
\frac{\partial L_d^e}{\partial q^{(\alpha)\, i}_u} &= \left. \sum_{k = 1}^n \left[ \sum_{\delta = 0}^{\gamma-\alpha-1} (-1)^{\delta} \frac{\mathrm{d}^{\delta}}{\mathrm{d} t^{\delta}} \left( \frac{\partial L}{\partial q^{(\alpha + \delta + 1)\, k}}(q(t), \dot{q}(t),...,q^{(\gamma)}(t)) \right)\right] \frac{\partial q^{(\alpha)\, k}}{\partial q^{(\alpha)\, i}_u}(t)\right\vert_{0}^{h}
\end{align*}
where $\alpha = 0,...,\gamma-1$, $i = 1, ..., n$ and $u = 0, 1$. Here and in the rest of the proof, the evaluations at $h$ and $0$ involve the entire expression.

Differentiating this expression with respect to $q^{(\beta)\, j}_v$,  $\beta = 0,...,\gamma-1$, and omitting arguments, we get
\begin{align*}
\frac{\partial^2 L_d^e}{\partial q^{(\alpha)\, i}_u \partial q^{(\beta)\, j}_v} &= \left. \sum_{k = 1}^n \left[ \sum_{\delta = 0}^{\gamma-\alpha-1} (-1)^{\delta} \frac{\partial}{\partial q^{(\beta)\, j}_v} \frac{\mathrm{d}^{\delta}}{\mathrm{d} t^{\delta}} \left( \frac{\partial L}{\partial q^{(\alpha + \delta + 1)\, k}}\right)\right] \frac{\partial q^{(\alpha)\, k}}{\partial q^{(\alpha)\, i}_u}\right\vert_{0}^{h}\\
&+ \left. \sum_{k = 1}^n \left[ \sum_{\delta = 0}^{\gamma-\alpha-1} (-1)^{\delta} \frac{\mathrm{d}^{\delta}}{\mathrm{d} t^{\delta}} \left( \frac{\partial L}{\partial q^{(\alpha + \delta + 1)\, k}} \right)\right] \frac{\partial^2 q^{(\alpha)\, k}}{\partial q^{(\alpha)\, i}_u \partial q^{(\beta)\, j}_v}\right\vert_{0}^{h}\,.
\end{align*}
All the terms in the second line of this equation vanish because the derivatives ${\partial q^{(\alpha)\, k}}/{\partial q^{(\alpha)\, i}_u}$ at $t=0$ are $\delta_{ki}\delta_{0u}$ and at $t=h$ they are $\delta_{ki}\delta_{1u}$ (see Lemma~\ref{lem:IVP_BVP}).

Separating the zeroth derivative term, commuting derivatives and expanding, we obtain
\begin{align*}
\frac{\partial^2 L_d^e}{\partial q^{(\alpha)\, i}_u \partial q^{(\beta)\, j}_v} &= \left. \sum_{k = 1}^n \left[ \sum_{\delta = 1}^{\gamma-\alpha-1} (-1)^{\delta} \frac{\mathrm{d}^{\delta}}{\mathrm{d} t^{\delta}} \left( \frac{\partial}{\partial q^{(\beta)\, j}_v} \frac{\partial L}{\partial q^{(\alpha + \delta + 1)\, k}}\right)\right] \frac{\partial q^{(\alpha)\, k}}{\partial q^{(\alpha)\, i}_u}\right\vert_{0}^{h}\\
&+ \left. \sum_{k = 1}^n \left[ \left( \frac{\partial}{\partial q^{(\beta)\, j}_v} \frac{\partial L}{\partial q^{(\alpha + 1)\, k}}\right)\right] \frac{\partial q^{(\alpha)\, k}}{\partial q^{(\alpha)\, i}_u}\right\vert_{0}^{h}\\
&= \left. \sum_{k = 1}^n \left[ \sum_{\delta = 1}^{\gamma-\alpha-1} (-1)^{\delta} \sum_{\ell = 1}^n \sum_{\epsilon = 0}^{\gamma} \left( \frac{\mathrm{d}^{\delta}}{\mathrm{d} t^{\delta}} \frac{\partial^2 L}{\partial q^{(\alpha + \delta + 1)\, k} \partial q^{(\epsilon)\, \ell}} \right) \frac{\partial q^{(\epsilon)\, \ell}}{\partial q^{(\beta)\, j}_v}  \right] \frac{\partial q^{(\alpha)\, k}}{\partial q^{(\alpha)\, i}_u}\right\vert_{0}^{h}\\
&+ \left. \sum_{k = 1}^n \left[ \sum_{\delta = 1}^{\gamma-\alpha-1} (-1)^{\delta} \sum_{\ell = 1}^n \sum_{\epsilon = 0}^{\gamma} \frac{\partial^2 L}{\partial q^{(\alpha + \delta + 1)\, k} \partial q^{(\epsilon)\, \ell}} \frac{\partial q^{(\epsilon + \delta)\, \ell}}{\partial q^{(\beta)\, j}_v} \right] \frac{\partial q^{(\alpha)\, k}}{\partial q^{(\alpha)\, i}_u}\right\vert_{0}^{h}\\
&+ \left. \sum_{k = 1}^n \left[ \sum_{\ell = 1}^n \sum_{\epsilon = 0}^{\gamma} \frac{\partial^2 L}{\partial q^{(\alpha + 1)\, k} \partial q^{(\epsilon)\, \ell}} \frac{\partial q^{(\epsilon)\, \ell}}{\partial q^{(\beta)\, j}_v} \right] \frac{\partial q^{(\alpha)\, k}}{\partial q^{(\alpha)\, i}_u}\right\vert_{0}^{h}
\end{align*}

First, we should notice that the total derivatives of the Lagrangian that appear in this formula are at most of order $\gamma - 1 < \gamma + 1$. Since the Lagrangian itself contains terms up to $q^{(\gamma)\,i}$, terms containing $q^{(2 \gamma - 1)\,i}$ will appear. However, since by the definition of $L_d^e$, $q^{(\alpha)}(t)$ is a solution of the Euler--Lagrange equations and, as shown in \cite{MR3562389}, these are $C^{2 \gamma}$ continuous, each possible term
\begin{equation*}
\left. \frac{\mathrm{d}^{\delta}}{\mathrm{d} t^{\delta}} \frac{\partial^2 L}{\partial q^{(\alpha + \delta + 1)\, k} \partial q^{(\epsilon)\, \ell}} \right\vert_{0}^h
\end{equation*}
will be a bounded function of the values $(q^{(0)} (0),...,q^{(2 \gamma - 1)}(0))$ and $(q^{(0)}(h),...,q^{(2 \gamma - 1)}(h))$.

Now we proceed to expand the expression for the Hessian in a Taylor series about an arbitrary point $\xi \in [0,h]$. From Lemma~\ref{lem:IVP_BVP} we deduce that $q^{(\alpha)\, i}(\xi) = (1 + \mathcal{O}(h)) q^{(\alpha)\, i}(0)$ and $q^{(\alpha)\, i}(\xi) = (1 + \mathcal{O}(h)) q^{(\alpha)\, i}(h)$. Moreover, from the same Lemma we know that the dominant terms are those multiplied by $\frac{\partial q^{(\zeta)\, \ell}}{\partial q^{(\beta)\, j}_v}$ with the highest $\zeta$ possible since these are $\mathcal{O}(h^{-\zeta+\beta})$. Clearly the lowest order terms occur for $\delta = \gamma - \alpha - 1$ and $\epsilon = \gamma$, and thus we may omit the rest:
\begin{align*}
\frac{\partial^2 L_d^e}{\partial q^{(\alpha)\, i}_u \partial q^{(\beta)\, j}_v} &= \left. \sum_{k = 1}^n \left[ (-1)^{\gamma-\alpha-1} \sum_{\ell = 1}^n \frac{\partial^2 L}{\partial q^{(\gamma)\, k} \partial q^{(\gamma)\, \ell}} \frac{\partial q^{(2 \gamma - \alpha - 1)\, \ell}}{\partial q^{(\beta)\, j}_v} + \mathcal{O}(h^{- 2 (\gamma - 1) + \alpha + \beta})\right] \frac{\partial q^{(\alpha)\, k}}{\partial q^{(\alpha)\, i}_u}\right\vert_{0}^{h}\\
&\phantom{=}+\text{h.o.t.}
\end{align*}

Noting the alternating signs and the fact that the derivatives appear in reversed order with respect to~\eqref{eq:y01_intermsof_x01}, the matrix $\mathbf{E}$ from that expression becomes $\mathbf{D}$ (as defined in the Appendix). Therefore we may expand the Hessian of $L_d^e$ as
\begin{align}
\left(
\begin{array}{cc}
\mathcal{A}_d & \mathcal{C}_d\\
\mathcal{C}_d^{\top} & \mathcal{B}_d
\end{array}
\right) &= \left(
\begin{array}{cc}
\mathbf{D} \mathbf{B}^{-1}(h) \mathbf{A}(h) & -\mathbf{D} \mathbf{B}^{-1}(h)\\
\mathbf{D} \mathbf{B}^{-1}(-h) & -\mathbf{D} \mathbf{B}^{-1}(-h)\mathbf{A}(-h)
\end{array}
\right) \otimes \mathcal{W}(\xi) + \text{h.o.t.}\label{eq:hessian_expansion}\\
&= \left(
\begin{array}{cc}
\mathbf{D} \mathbf{C}^{-1}(h) \mathbf{D} & -\mathbf{D} \mathbf{B}^{-1}(h)\\
\mathbf{D} \mathbf{B}^{-1}(-h) & -\mathbf{D} \mathbf{C}^{-1}(-h) \mathbf{D}
\end{array}
\right) \otimes \mathcal{W}(\xi) + \text{h.o.t.}\,,\nonumber
\end{align}
where $\mathcal{W}(\xi) := \mathcal{W}(q^{[\gamma]}(\xi))$. Thus, from Corollary \ref{cor:B_matrix}, we get that
\begin{equation*}
\det \mathcal{C}_d = \left( - h^{- \gamma^2} \prod_{\alpha = 0}^{\gamma - 1} \frac{(\gamma + \alpha)!}{\alpha!} \right)^n \left( \det \mathcal{W}(\xi) \right)^\gamma + \mathcal{O}(h^{- n \gamma^2 + 1})\,,
\end{equation*}
and therefore, the regularity of $L_d^e$ follows from the regularity of $L$.

Similarly,
\begin{align*}
\det \mathcal{A}_d &= \left( h^{- \gamma^2} \prod_{\alpha = 0}^{\gamma - 1} \frac{(\gamma + \alpha)!}{\alpha!} \right)^n \left( \det \mathcal{W}(\xi) \right)^\gamma + \mathcal{O}(h^{- n \gamma^2 + 1})\,,\\
\det \mathcal{B}_d &= \left( h^{- \gamma^2} \prod_{\alpha = 0}^{\gamma - 1} \frac{(\gamma + \alpha)!}{\alpha!} \right)^n \left( \det \mathcal{W}(\xi) \right)^\gamma + \mathcal{O}(h^{- n \gamma^2 + 1})\,.
\end{align*}

Using Lemma \ref{lem:C_matrix}.\ref{lem:C_matrix.itm:C_symm}, $-\mathbf{D} \mathbf{C}^{-1}(-h)\mathbf{D} = \mathbf{C}^{-1}(h)$ and applying Sylvester's law of inertia to $\mathbf{D} \mathbf{C}^{-1}(h) \mathbf{D}$, both matrices are definite. Thus, the same must be true for $\mathcal{A}_d$ and $\mathcal{B}_d$ for any $h \in U^*_0$, and in particular positive-definite for $U^*_{0,+}$.

Finally, let us denote by
\begin{equation*}
    \left(
    \begin{array}{cc}
        \mathcal{A}_{d, 0} & \mathcal{C}_{d, 0}\\
        \mathcal{C}_{d, 0}^{\top} & \mathcal{B}_{d, 0}
    \end{array}
\right)
\end{equation*}
the matrix of leading terms in Eq.~\eqref{eq:hessian_expansion}, and consider the following factorization (see \cite{MR2061575}):
\begin{multline*}
\left(
\begin{array}{cc}
\mathcal{A}_{d, 0} & \mathcal{C}_{d, 0}\\
\mathcal{C}_{d, 0}^{\top} & \mathcal{B}_{d, 0}
\end{array}
\right) =\\ 
\left(
\begin{array}{cc}
I & 0\\
\mathcal{C}_{d, 0}^{\top} (\mathcal{A}_{d, 0})^{-1} & I
\end{array}
\right) \left(
\begin{array}{cc}
\mathcal{A}_{d, 0} & 0\\
0 & \mathcal{B}_{d, 0} - \mathcal{C}_{d, 0}^{\top} (\mathcal{A}_{d, 0})^{-1} \mathcal{C}_{d, 0}
\end{array}
\right) \left(
\begin{array}{cc}
I & (\mathcal{A}_{d, 0})^{-1} \mathcal{C}_{d, 0}\\
0 & I
\end{array}
\right)
\end{multline*}
Let us show that
\begin{equation*}
\mathcal{B}_{d, 0} = \mathcal{C}_{d, 0}^{\top} (\mathcal{A}_{d, 0})^{-1} \mathcal{C}_{d, 0}\,.
\end{equation*}
Indeed, from their definitions and Lemma \ref{cor:AB_matrix},
\begin{align*}
\mathcal{C}_{d, 0}^{\top} (\mathcal{A}_{d, 0})^{-1} \mathcal{C}_{d, 0} &= - \mathbf{D} \mathbf{B}^{-1}(-h) \mathbf{A}^{-1}(h) \mathbf{B}(h) \mathbf{D} \mathbf{D} \mathbf{B}^{-1}(h) \otimes \mathcal{W}(\xi) \mathcal{W}^{-1}(\xi) \mathcal{W}(\xi)\\
&= - \mathbf{D} \mathbf{B}^{-1}(-h) \mathbf{A}^{-1}(h) \otimes \mathcal{W}(\xi)=- \mathbf{D} \mathbf{B}^{-1}(-h) \mathbf{A}(-h) \otimes \mathcal{W}(\xi)=\mathcal{B}_{d, 0} \,.
\end{align*}
Since the block diagonal matrix obtained through this factorization and the original are symmetric and congruent, Sylvester's law of inertia tells us that their signature and number of zeros in the spectrum will coincide. That is, this matrix has $\gamma \dim Q$ positive eigenvalues and a zero eigenvalue with multiplicity  $\gamma \dim Q$ and, therefore, they are positive semidefinite.
\end{proof}

\begin{remark}
This result generalizes Theorem 4.2 in \cite{MR3562389}.
\end{remark}

Now, we will use these results to obtain the convergence of the parallel algorithm for a  general  discrete Lagrangian $L_d\colon T^{(\gamma-1)}Q\times T^{(\gamma-1)}Q\rightarrow {\mathbb R}$ comparing it with the exact discrete Lagrangian. However, special care must be taken with the approximation of the successive derivatives between a discrete Lagrangian and the exact Lagrangian as summarized in the following remark.
\begin{remark}
In general, as it is shown in \cite{patrick},  the fact that a discrete $\gamma$-th order Lagrangian $L_d$ be consistent to order $r$ does not imply that its discrete fibre derivative $\mathbb{F}L_d^{\pm}\colon T^{(\gamma-1)} Q \times T^{(\gamma-1)} Q \to T^*(T^{(\gamma - 1)}Q)$ be of order $r$. Here, we denote by 
$\mathbb{F}L_d^{\pm}$ the discrete Legendre transformations defined by \cite{MarsdenWest_ActaNum}:
\begin{align*}
\mathbb{F}L_d^{-}(q_0^{[\gamma-1]},q_1^{[\gamma-1]})&=
( q_0^{[\gamma-1]}, -D_1 L_d (q_0^{[\gamma-1]},q_1^{[\gamma-1]}))\\
\mathbb{F}L_d^{+}(q_1^{[\gamma-1]},q_1^{[\gamma-1]})&=
( q_0^{[\gamma-1]}, D_2 L_d (q_0^{[\gamma-1]},q_1^{[\gamma-1]}))
\end{align*}
That is, if $F_L^{h}\colon T^{(2 \gamma - 1)}Q \to T^{(2 \gamma -1)}Q$ denotes the Lagrangian flow derived from the Euler--Lagrange equations, $\pi^{\alpha}_{\beta}\colon T^{(\alpha)} Q \to T^{(\beta)}Q$, with $\alpha > \beta$ denotes the natural projection $( q,\dot{q},...,q^{(\beta)}, ..., q^{(\alpha)}) \mapsto ( q,\dot{q},...,q^{(\beta)} )$, and $q_{0}^{[2 \gamma -1]} := ( q_0,\dot{q}_0,...,q_0^{(2 \gamma - 1)} )$,
\begin{align*}
\mathbb{F}L_d^{-}(q_{0}^{[\gamma -1]},( \pi^{2 \gamma - 1}_{\gamma - 1} \circ F_L^h)(q_{0}^{[2 \gamma -1]})) &\neq \mathbb{F}L(q_{0}^{[2 \gamma -1]}) + \mathcal{O}(h^{r+1})\\
\mathbb{F}L_d^{+}(q_{0}^{[\gamma -1]},( \pi^{2 \gamma - 1}_{\gamma - 1} \circ F_L^h)(q_{0}^{[2 \gamma -1]})) &\neq ( \mathbb{F}L \circ F_L^h)(q_{0}^{[2 \gamma -1]}) + \mathcal{O}(h^{r+1})
\end{align*}
where $\mathbb{F}L\colon  T^{(2k-1)}Q\rightarrow T^*(T^{(\gamma-1)}Q)$ is the Legendre transformation of $L$ (see \cite{Generalized-classical}). 
\end{remark}

\addtocounter{example}{-1}
\begin{example}[continued]
One can check that the discrete fibre derivatives at $h = 0$ converge to
\begin{align*}
\mathbb{F}L_d^{\alpha,\, \pm}(q_{0}^{[1]},q_{0}^{[1]}) = \left.\left(q^i, \dot{q}^i, - 3 \alpha \left[ \frac{\mathrm{d}}{\mathrm{d} t}\left( \frac{\partial L}{\partial \ddot{q}^i} \right) + (\alpha - 1) \frac{\partial^2 L}{\partial \ddot{q}^i \partial \ddot{q}^j} q^{(3) \, j} \right], \frac{\partial L}{\partial \ddot{q}^i} \right)\right\vert_{q_0^{[3]}}
\end{align*}
which do not coincide with
\begin{equation*}
\mathbb{F}L(q_0^{[3]}) = \left.\left(q^i, \dot{q}^i, \frac{\partial L}{\partial \dot{q}^i} - \frac{\mathrm{d}}{\mathrm{d} t}\left( \frac{\partial L}{\partial \ddot{q}^i} \right), \frac{\partial L}{\partial \ddot{q}^i} \right)\right\vert_{q_0^{[3]}}\,.
\end{equation*}
for any $\alpha$, despite $L_d^{\alpha}$ being order $2$.
\end{example}

With this in mind we need the following proposition.
\begin{proposition}\label{prop-order}
If $L_d\colon T^{(\gamma-1)}Q\times T^{(\gamma-1)}Q\rightarrow {\mathbb R}$
is a discrete Lagrangian that is consistent to order $r$ with a Lagrangian $L\colon T^{(\gamma)}Q\rightarrow {\mathbb R}$ then 
\begin{align}\label{hessian-order}
\frac{\partial^2 L_d}{\partial q^{(\alpha)\, i}_u \partial q^{(\beta)\, j}_v} 
=\frac{\partial^2 L^e_d}{\partial q^{(\alpha)\, i}_u \partial q^{(\beta)\, j}_v}+{\mathcal O}(h^{r+1-2\gamma}), \quad 1\leq \alpha,\beta\leq \gamma-1
\end{align}
\end{proposition}
\begin{proof}
The most general form of a  discrete Lagrangian consistent to order $r$ is
\begin{equation*}
L_d(q_0^{[\gamma - 1]},q_1^{[\gamma - 1]}) = L_d^e(q_0^{[\gamma - 1]},q_1^{[\gamma - 1]}) + h^{r + 1} g\left(q^{[2\gamma-1]}(q_0^{[\gamma - 1]},q_1^{[\gamma - 1]},h)\right) + \mathcal{O}(h^{r + 1})
\end{equation*}
where $t\rightarrow q^{[2\gamma-1]}(q_0^{[\gamma - 1]},q_1^{[\gamma - 1]},t)$ denotes the curve in $T^{(2\gamma-1)}Q$ solution of the Euler--Lagrange equations for $L$ with boundary conditions $q_0^{[\gamma - 1]},q_1^{[\gamma - 1]}$, and 
$g\colon T^{(\gamma)}Q \to \mathbb{R}$ is an arbitrary $C^{2}$ function.
Using Lemma  \ref{lem:IVP_BVP}, $q_0^{(\gamma)} \sim h^{-{\gamma}} f(q_0^{[\gamma - 1]},q_1^{[\gamma - 1]},h)$ and thus, for $u = 0, 1$,
\begin{align*}
D_u L_d(q_0^{[\gamma - 1]},q_1^{[\gamma - 1]}) &= D_u L_d^e(q_0^{[\gamma - 1]},q_1^{[\gamma - 1]}) + h^{r + 1 - \gamma} g'\left(q_0^{[\gamma]}\right) f(q_0^{[\gamma - 1]},q_1^{[\gamma - 1]},h) + \mathcal{O}(h^{r + 1})\\
&= D_u L_d^e(q_0^{[\gamma - 1]},q_1^{[\gamma - 1]}) + \mathcal{O}(h^{r + 1 - \gamma})\,.
\end{align*}
Therefore, the fibre derivative may only be guaranteed to be of order $r-\gamma$. 
Now taking an additional derivative, the order could decay by an additional $-\gamma$ and, as a consequence, we deduce Equation~\eqref{hessian-order}.
\end{proof} 

\begin{theorem}\label{thm:Ld-from-L-defpos}
If $L_d\colon T^{(\gamma-1)}Q\times T^{(\gamma-1)}Q \to \mathbb{R}$ is a discrete Lagrangian that is consistent to order $2\gamma - 1$ with a Lagrangian $L\colon T^{(\gamma)}Q \to \mathbb{R}$ with a positive-definite Hessian matrix $\mathcal{W}$, then $L_d$ satisfies all the hypotheses in Theorem~\ref{prp:discrete_positive_def}, in a deleted neighborhood of $h = 0$, and hence the block Jacobi and 1-step block Jacobi--Newton methods converge locally.
\end{theorem}
\begin{proof}
If $L_d\colon T^{(\gamma-1)}Q\times T^{(\gamma-1)}Q \to \mathbb{R}$ is a discrete Lagrangian that is consistent to order $2\gamma - 1$ with a Lagrangian $L\colon T^{(\gamma)}Q \to \mathbb{R}$ 
then from Proposition \ref{prop-order} we will have that 
\begin{align*}
\frac{\partial^2 L_d}{\partial q^{(\alpha)\, i}_u \partial q^{(\beta)\, j}_v} 
=\frac{\partial^2 L^e_d}{\partial q^{(\alpha)\, i}_u \partial q^{(\beta)\, j}_v}+{\mathcal O}(1), \quad 1\leq \alpha,\beta\leq \gamma-1
\end{align*}
but from Proposition \ref{prp:regularity} we immediately deduce that
\[
\mathrm{H}^{L_d}=\mathrm{H}^{L_d^e}+{\mathcal O}(1).
\]
Therefore, $\mathrm{H}^{L_d}$ is positive-definite for small enough $h$, which implies the local convergence of the block Jacobi and 1-step block Jacobi--Newton methods.
\end{proof}

As we have seen in Example~\ref{exp:trapezoidal} given an arbitrary discrete Lagrangian, in order to ensure the convergence of our method one can perform a direct check on its Hessian. However, knowing the order of the discrete Lagrangian we can automatically obtain this convergence by invoking Theorem~\ref{thm:Ld-from-L-defpos}. 

\begin{example}
\label{exp:2stageGauss}
Consider the discrete second order Lagrangian based on the 2-stage Gauss method given by
\begin{equation*}
L_d(q_0,v_0,q_1,v_1) = \frac{h}{2} \left[ L\left(Q_1,V_1,A_1\right) + L\left(Q_2,V_2,A_2\right) \right]\,.
\end{equation*}
with
\begin{align*}
Q_1 &:= \left(\frac{1}{2} - \frac{\sqrt{3}}{6}\right) q_1 + \left(\frac{1}{2} + \frac{\sqrt{3}}{6}\right) q_0 - \frac{h}{12} \left(v_1 - v_0\right)\,,\\
Q_2 &:= \left(\frac{1}{2} + \frac{\sqrt{3}}{6}\right) q_1 + \left(\frac{1}{2} - \frac{\sqrt{3}}{6}\right) q_0 - \frac{h}{12} \left(v_1 - v_0\right)\,,\\
V_1 &:= \frac{q_1 - q_0}{h} - \frac{\sqrt{3}}{6} \left(v_1 - v_0\right)\,,\\
V_2 &:= \frac{q_1 - q_0}{h} + \frac{\sqrt{3}}{6} \left(v_1 - v_0\right)\,,\\
A_1 &:= \frac{1}{h}\left[(1 - \sqrt{3}) v_1 - (1 + \sqrt{3}) v_0\right] + \frac{2 \alpha}{h^2} (x_1 - x_0)\,,\\
A_2 &:= \frac{1}{h}\left[(1 + \sqrt{3}) v_1 - (1 - \sqrt{3}) v_0\right] - \frac{2 \alpha}{h^2} (q_1 - q_0)\,.
\end{align*}

It can be shown that this discrete Lagrangian is consistent to order $4$ for any regular Lagrangian $L$ (in the sense of Definition~\ref{def:order-higher}). Thus, if $L$ has a positive-definite Hessian matrix, we know that the discrete Lagrangian $L_d$ satisfies all the hypothesis of Theorem~\ref{thm:Ld-from-L-defpos} for small enough $h$. 
\end{example}

\begin{remark}
	It is rather straightforward to generalize this result to variable-step discretizations, with $h^* \geq \max_k h_k$ and $h_k = t_{k+1}-t_k$ for $k = 0,...,N-1$.
\end{remark}

\begin{remark}
Theorem~\ref{thm:Ld-from-L-defpos} gives us a way to automatically guarantee that the conditions of Theorem~\ref{prp:discrete_positive_def} are satisfied, but it is quite restrictive in terms of the order of the discretization.

However, provided the continuous Lagrangian is regular and positive-definite, Proposition~\ref{prp:regularity} offers a simplified criterion to determine whether an arbitrary discretization scheme will lead to a discrete Lagrangian satisfying the conditions of Theorem~\ref{prp:discrete_positive_def}. From the proof of the proposition we see that it suffices to ensure that the matrix of coefficients of the leading terms of the Hessian, which depends solely on the discretization, is regular and positive-definite. Thus, the conditions will be satisfied if and only if the resulting discrete Lagrangian of any other continuous Lagrangian with the same properties satisfies the conditions of Theorem~\ref{prp:discrete_positive_def}. In particular one may check with the simple Lagrangian $L (q^{[\gamma]}) = \frac{1}{2} \left\Vert q^{(\gamma)} \right\Vert^2$.
\end{remark}

\subsection{Convergence for time-dependent Lagrangian systems}
In this section, we will briefly show that the results about convergence  of the parallel iterative method for variational integrators also apply for time-dependent Lagrangian systems. 

Consider a regular $C^{2\gamma}$ Lagrangian $L\colon {\mathbb R} \times T^{(\gamma)}Q\rightarrow {\mathbb R}$ and define the exact discrete Lagrangian of the $k$-th interval, $[kh, (k+1)h]$, by analogy with the autonomous case (see \eqref{exact-gamma}): 
\begin{equation}\label{exact-gamma-time dependent}
L_{d,k}^{e}(q^{[\gamma-1]}_k, q^{[\gamma-1]}_{k+1})=
\int^{(k+1)h}_{kh} L(t, q^{[\gamma]}(t))\; dt, \quad 0\leq k\leq N-1\,. 
\end{equation}
Here $q\colon [kh, (k+1)h]\rightarrow Q$ denotes the unique $C^{2 \gamma}$ solution curve of the Euler--Lagrange equations for $L$ satisfying the boundary conditions 
$q^{[\gamma-1]}(kh)=q^{[\gamma-1]}_k$ and $q^{[\gamma-1]}((k+1)h)=q^{[\gamma-1]}_{k+1}$. 

By $L_{d,k}\colon T^{(\gamma-1)} Q \times T^{(\gamma-1)} Q \to \mathbb{R}$ we denote discrete $\gamma$-th order Lagrangians that are approximations of the exact discrete Lagrangians by extending Definition~\ref{def:order-higher}. 

The parallel method is easily adapted to this case considering the following modification of the algorithm proposed in Equation~\eqref{eq:pDEL} 
\begin{equation}\label{eq:pDEL-time-dependent}
  D_2L_{d, k-1}(q^{[\gamma-1]}_{k-1},\bar q^{[\gamma-1]}_k)+D_1L_{d,k}(\bar q^{[\gamma-1]}_k,q^{[\gamma-1]}_{k+1})=0, \qquad 1\leq k\leq N-1\, .  
\end{equation}
It is straightforward to check that Theorem~\ref{prp:discrete_positive_def} and Theorem~\ref{thm:Ld-from-L-defpos} are still valid in the time-dependent case provided that the regularity and definiteness conditions hold in the time interval under consideration. We will see some applications to navigation problems and an astrodynamical example where it is necessary to resort to time-dependent Lagrangians and their discretization (see Subsection~\ref{examples-time-dependent}).

\section{Examples}\label{sec-examples}

\subsection{Zermelo's navigation problem}\label{section-zermelo}

Zermelo's navigation problem \cite{Zermelo, Bao, Kopacz, java} is usually presented as a time-optimal control problem, which aims to find the minimum time trajectories on a Riemannian manifold $(Q, g)$ under the influence of a drift vector field $W\in {\mathfrak X}(Q)$, which can be interpreted as wind (or water currents). The goal is to find a navigation path $\gamma(s)$, for $s$ in a given interval $[s_0,s_N]$, joining two points in $Q$ in the shortest possible time in the presence of this wind. The physical time $t$ is related to $s$ via an increasing function $t(s)$ that can be computed from $\gamma$ (see Section~\ref{sec:zermelotimedependent}). We assume that the ship engine provides a constant thrust relative to the wind, that is, $\big|\frac{d\gamma}{dt}(t(s))-W(\gamma(s))\big|=1$, where $|\cdot|$ denotes the norm provided by $g$. It is also assumed that $|W(q)|< 1$ for all $q\in Q$.

These minimum time trajectories are precisely the geodesics for a particular type of Finsler metric, a Randers metric defined by (see \cite{Bao} and references therein)
\begin{equation*}
F(q,v_q)=\sqrt{a(v_q, v_q)}+\langle b(q), v_q\rangle
\end{equation*}
where 
\begin{align*}
a(v_q, v_q)&=\frac{1}{\alpha(q)}g(v_q,v_q)+\frac{1}{\alpha(q)^2}g(W(q), v_q)^2\\
\langle b(q), v_q\rangle&=-\frac{1}{\alpha(q)} g(W(q), v_q)=-\left\langle \frac{\flat_g({W}(q))}{\alpha(q)}, v_q\right\rangle \\
\alpha(q)&=1-g(W(q),W(q))=1-|W(q)|^2>0.
\end{align*}
Here $\flat_g\colon {\mathfrak X}(Q)\rightarrow \Omega^1(Q)$ is the musical isomorphism defined by $\langle \flat_g(X), Y\rangle=g(X,Y)$ for all $X, Y\in {\mathfrak X}(Q)$. 

The time it takes the ship to move along a curve $\gamma\colon [s_0, s_N]\to Q$ is
\begin{equation}\label{eq:intF}
\int_{s_0}^{s_N} F(\gamma(s), \dot \gamma(s))\,ds.
\end{equation}
Note that this integral is invariant under orientation-preserving reparametrizations of $\gamma$, since Finsler metrics are positively 1-homogeneous, that is, $F(q,\lambda v_q)=\lambda F(q,v_q)$ for any $\lambda>0$. Therefore, the solution curves are not unique; in fact, $F$ is not regular as a Lagrangian function. Similar to the case of Riemannian metrics and the problem of minimizing length or energy, this can be circumvented by considering instead the functional
\begin{equation}\label{eq:intF2}
\int_{s_0}^{s_N} \big(F(\gamma(s),\dot\gamma(s))\big)^2\,ds .
\end{equation}
Any extremal of this functional will be an extremal of \eqref{eq:intF}, and any extremal of \eqref{eq:intF} admits an orientation-preserving reparametrization that makes it an extremal of \eqref{eq:intF2} (see \cite{Masiello} and references therein).

As a particular case, consider $Q={\mathbb R}^2$ with the Euclidean metric, where we are to find critical curves $(x,y)=(x(s),y(s))$ for the functional
\begin{equation}\label{eq:functionalFinslerR2}
\int_{s_0}^{s_N}\left[\sqrt{\frac{1}{\alpha}(\dot{x}^2+\dot{y}^2)+\frac{1}{\alpha^2}(W_1(x,y)\dot{x}+W_2(x, y)\dot{y})^2}-\frac{1}{\alpha} \left(W_1(x, y)\dot{x}+W_2(x, y)\dot{y}\right)\right]^2\, ds
\end{equation}
with $\alpha=1-(W_1^2+W_2^2)$. Here $\dot x=dx/ds$, $\dot y=dy/ds$.

Figure \ref{fig:timeoptimal} shows six local solutions to the navigation problem found using our approach, starting at $(0,0)$ and ending at $(6,2)$, for the vector field
$W=1.7\cdot(-R_{2,2}-R_{4,4}-R_{2,5}+R_{5,1})$, where
\[
R_{a,b}(x,y)=\frac{1}{3((x-a)^2+(y-b)^2)+1}\begin{pmatrix}
-(y-b)\\x-a
\end{pmatrix}.
\]
The scale factor $1.7$ was chosen so that the maximum value of $|W|$ is almost 1.

For this example, we used the Jacobi-Newton method \eqref{eq:NewtonRaphson} with the discrete Lagrangian
\[
L_d(q_0,q_1)=\frac{h}{2} \left(F^2\left(q_0,\frac{q_1-q_0}{h}\right)+F^2\left(q_1,\frac{q_1-q_0}{h}\right)\right).
\]
These solutions were obtained using different piecewise straight lines as initial guesses for the trajectories and with $N=80$.

The total navigation time~\eqref{eq:intF}, which is displayed beside each trajectory, is locally optimal. In general, finding a global minimum will require exploring different initial guesses.

\begin{figure}
	\begin{center}
		\includegraphics[scale=.7]{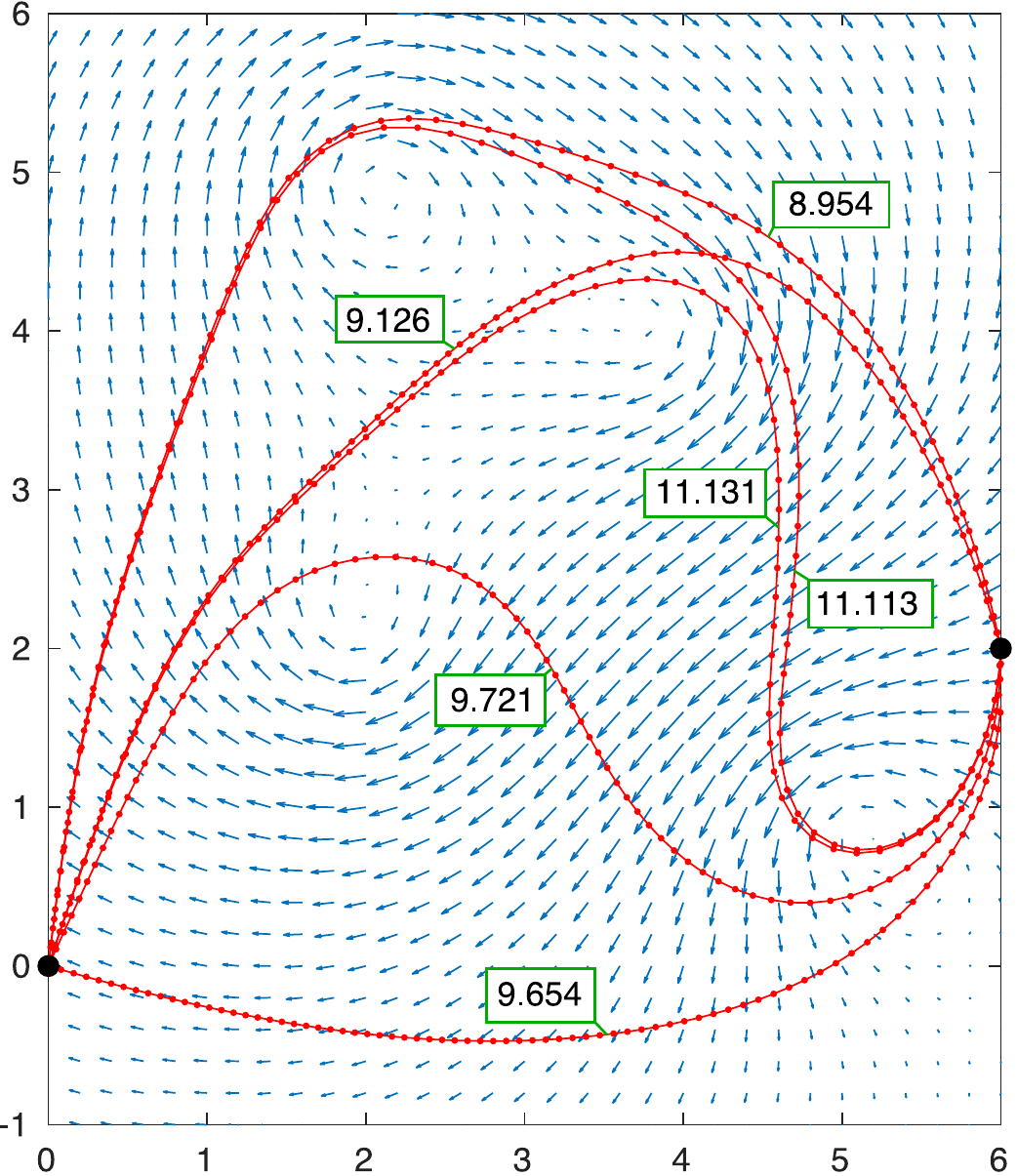}
	\end{center}
	\caption{Several local solutions to the optimal time navigation problem starting from $(0,0)$ and ending at $(6,2)$. The time for each trajectory is shown.}\label{fig:timeoptimal}
\end{figure}

\subsection{Fuel-optimal navigation problem}\label{fuel}
We also consider a non-equivalent variant of Zermelo's problem. If $T>0$ is a fixed time, we seek trajectories minimizing the cost function
\begin{equation*}
\int_0^T \frac{1}{2}(u_1^2+u_2^2)\; dt\,,
\end{equation*}
which can be interpreted as a measure of fuel expenditure.
The system is subject to the control equations
\begin{align*}
\dot{x}&= u_1 + W_1(x, y)\,,\\
\dot{y}&= u_2 + W_2(x, y)\,.
\end{align*}
The goal is to arrive at a given destination at time $T$, extremizing fuel expenditure with no a priori bounds on the engine's power. 
This problem is equivalent to solving the Euler--Lagrange equations for the Lagrangian
\begin{equation*}
L(x, y, \dot{x}, \dot{y}) = \frac{1}{2} \left[(\dot{x}-W_1(x,y))^2+ (\dot{y}-W_2(x,y))^2\right]\,,
\end{equation*}
with fixed $(x(0), y(0))$ and $(x(T), y(T))$ as boundary conditions.

For our simulations, we considered $W(x,y)=(
\cos(2x-y-6),
\frac{2}{3}\sin(y)+x-3)$.
We discretized the Lagrangian as 
\[
L_d(q_0,q_1)=\frac{h}{2} \left(L\left(q_0,\frac{q_1-q_0}{h}\right)+L\left(q_1,\frac{q_1-q_0}{h}\right)\right).
\]

\begin{figure}
	\begin{center}
		\includegraphics[scale=.7]{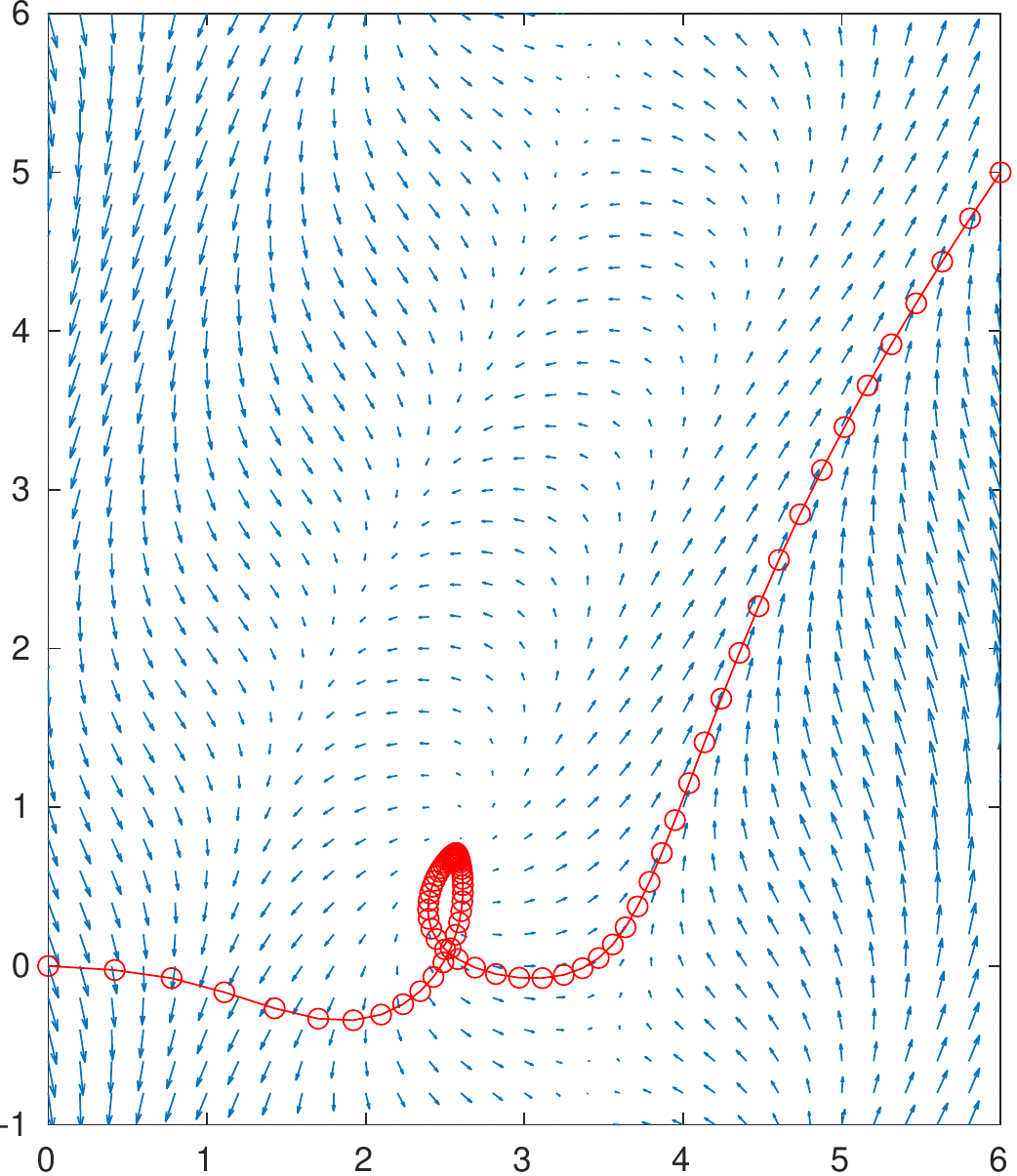}
	\end{center}
	\caption{A minimal fuel trajectory for a fixed total duration $T=30$, joining $(0,0)$ to $(6,5)$, with $N=200$. }\label{fig:fueloptimal}
\end{figure}

Figure \ref{fig:fueloptimal} shows a trajectory that has a locally optimal fuel expenditure among the discrete curves joining the given points $(0,0)$ and $(6,5)$ in time $T=30$. A straight line was used as the initial guess. Notice that in the first part of its journey, the ship travels to an equilibrium point of $W$, where it awaits the right moment to continue to its destination, which must be reached at the specified time. We emphasize that in this variation of the problem the total travel time is imposed externally. Other values for $T$ will have optimal trajectories with different fuel expenditure. For $T<9$ (approximately) the optimal trajectories do not pass near the equilibrium mentioned above.

\subsection{Interpolation problems}\label{sec:interpolationzermelo}
In this section we explore the application of our parallel iterative method to the case of a second-order Lagrangian system (see Section~\ref{sec:variationaldiscreteequations}) in the context of interpolation problems. For instance,  this kind of problems appear when comparing a series of images in longitudinal studies \cite{crouch-leite,Invariant1,Invariant2}.  Let $N\in \mathbb{N}$, $[t_0,t_N]\subset \mathbb{R}$, and $h=(t_N-t_0)/N$.  
Assume that we have $l+1$ interpolation points or knots $\hat{q}_a\in Q$, $a=0,\dots ,l$, which are reached at times 
$\hat{t}_a =t_0+ N_a h$,
where $N_a\in \{0, 1, \ldots, N\}$,
$N_a< N_b$ if $a<b$, with $N_0=0$ and $N_l=N$.

In order to discretize this problem, we replace the integral~\eqref{ElasticSplines} by a sum over times $t_k=t_0+kh$ for $k=0, \ldots, N$. 
Following our approach in \cite{MR3562389}, we discretize the action as
\begin{equation} \label{ElasticSplines-d}
\mathcal{J}_d := \sum_{k=0}^{N-1} L_d(q_k, v_k, q_{k+1}, v_{k+1})
\end{equation}
where $L_d\colon TQ\times TQ\rightarrow {\mathbb R}$ is a discretization of $L$.
Moreover, the problem is subject to the interpolation constraints
\begin{equation}\label{aq1}
q_{N_a}=\hat{q}_a , \quad\text{for all }a=1,\ldots, l-1
\end{equation}
and the boundary conditions
\begin{equation}\label{aq2}
q_0=\hat{q}_0,\quad v_0= \hat{v}_0 ,\quad\text{and}\quad
q_N=\hat{q}_l,\quad v_N= \hat{v}_l .
\end{equation}

Our parallel integrator works as follows. Take an arbitrary sequence $\{(q_k, v_k)\}$ satisfying the interpolation constraints \eqref{aq1} and the boundary conditions \eqref{aq2}.
Now construct the sequence $\{(\bar{q}_k, \bar{v}_k)\}$ by solving the parallelized problem 
\begin{align}\label{eq1-interpola}
D_3L_d(q_{k-1}, v_{k-1}, \bar{q}_{k}, \bar{v}_k)
+D_1L_d(\bar{q}_{k}, \bar{v}_k ,q_{k+1}, v_{k+1})&=0\; ,\\
D_4L_d(q_{k-1}, v_{k-1}, \bar{q}_{k}, \bar{v}_k)
+D_2L_d(\bar{q}_{k}, \bar{v}_k ,q_{k+1}, v_{k+1})&=0
\end{align}
if $1\leq k\leq {N}-1$ and $k\not=N_a$, $1\leq a \leq l - 1$. At each knot $k=N_a$, $1\leq a\leq l-1$, take $\bar{q}_{N_a}=q_{N_a}$ and compute $\bar{v}_{N_a}$ by solving the equation
\begin{equation}\label{eq2-interpolation}
D_4L_d(q_{N_a-1}, v_{N_a-1}, \bar{q}_{N_a}, \bar{v}_{N_a})
+D_2L_d(\bar{q}_{N_a}, \bar{v}_{N_a}, q_{N_a+1}, v_{N_a+1})=0\; .
\end{equation}
Finally, take $(\bar{q}_0,\bar{v}_0)=(q_0,v_0)$, $(\bar{q}_N,\bar{v}_N)=(q_N,v_N)$. 
Observe that the derived sequence $\{(\bar{q}_k, \bar{v}_k)\}$, $k=0, \ldots, N$
satisfies the 
interpolation constraints
\[
\bar{q}_{N_a}=\hat{q}_a , \quad\text{for all } a=1,\ldots, l-1
\]
and the boundary conditions 
\begin{equation*}
\bar q_0=\hat{q}_0,\quad \bar v_0= \hat{v}_0 ,\quad\text{and}\quad
\bar q_N=\hat{q}_l,\quad \bar v_N= \hat{v}_l \; .
\end{equation*}
The convergence results given in Section \ref{sec:convergence} are directly applicable to this case. In fact, the corresponding Hessian matrix for the interpolation problem is simply the restriction of the Hessian matrix $\mathrm{H}(q_d^{[\gamma-1]})$ to a subspace. Then, if the unrestricted Hessian is positive-definite, so will be any restriction to a subspace. Therefore, by iterating this procedure, we approach a trajectory having a locally optimum value of the cost functional \eqref{ElasticSplines-d}.

\subsubsection{An application: fuel-optimal control problem with a weight minimizing the total variation in the control variables }\label{example:interpolation}
As a modification of the application given in Section \ref{fuel}, consider the following optimal control problem. Our aim is still minimizing the fuel expenditure functional while also minimizing the total variation in the control variables. 
Now the goal is to navigate from a departure point to a destination point passing through given waypoints (knots) at prescribed times, minimizing the cost functional
\[
\int_0^T \frac{1}{2}(u_1^2+u_2^2 +c v_1^2+cv_2^2)\; dt
\]
subject to the control equations
\begin{equation*}
\begin{array}{rclrcl}
\dot{x} &=& u_1+W_1(x, y), & \quad\dot{y} &=& u_2+W_2(x,y),\\
\dot{u}_1&=&v_1, & \dot{u}_2&=&v_2\,.
\end{array}
\end{equation*}
Here $c > 0$ is a weight.

The continuous problem is equivalent to solving the fourth-order Euler--Lagrange equations for the second-order Lagrangian
\begin{multline*}
L(x, y, \dot{x}, \dot{y}, \ddot{x}, \ddot{y})=\frac{1}{2} \left[(\dot{x}-W_1(x,y))^2+ (\dot{y}-W_2(x,y))^2\right.\\	
+c\left(\ddot{x}- D_1W_1(x(t),y(t))\dot{x}- 
D_2W_1(x(t),y(t))\dot{y}\right)^2
\\
\left.
+c\left(\ddot{y}- D_1W_2(x(t),y(t))\dot{x}- 
D_2W_2(x(t),y(t))\dot{y}\right)^2
\right]
\end{multline*}
As boundary conditions, we consider $(q(0), \dot q(0))$ and $(q(T), \dot q(T))$ fixed. In addition, the system is subject to the
interpolation constraints
\begin{equation}\label{aq1example}
q(\hat{t}_a)=\hat{q}_a , \quad\text{for all }a=1,\ldots, l-1
\end{equation}
with $0<\hat{t}_a<\hat{t}_b<T$ for all $a, b\in\{1,\ldots, l-1\}$ and $a<b$.

As a discretization of the cost function we propose, for instance, a 2-stage Lobatto discretization (see Example \ref{exp:trapezoidal}):
\begin{align*}
&L_d(q_k, v_k, q_{k+1}, v_{k+1}) =\\
&\quad\frac{h}{2} \left[ L\left(q_k, v_k, \frac{2}{h^2}(3(q_{k+1}-q_k)-h(v_{k+1}+2v_k)\right) \right.\\
&\quad + \left. L\left(q_{k+1}, v_{k+1}, -\frac{2}{h^2}(3(q_{k+1}-q_k)-h(2v_{k+1}+v_k)\right)\right]
\end{align*}

In Figure \ref{fig:secondorder} we show an optimal trajectory starting at $(0,0)$ and ending at $(3,5)$ at $T=60$, with zero velocity at both endpoints, and passing through the prescribed positions $(1,3)$ and $(5,2)$ at times $20$ and $40$ respectively. The vector field $W$ is the same as in the previous example. We used $c=50$ and $N=240$.

All of the examples have been computed using the Jacobi-Newton method mentioned in Section~\ref{sec:parallel_approach}.

\begin{figure}
	\begin{center}
		\includegraphics[scale=.7]{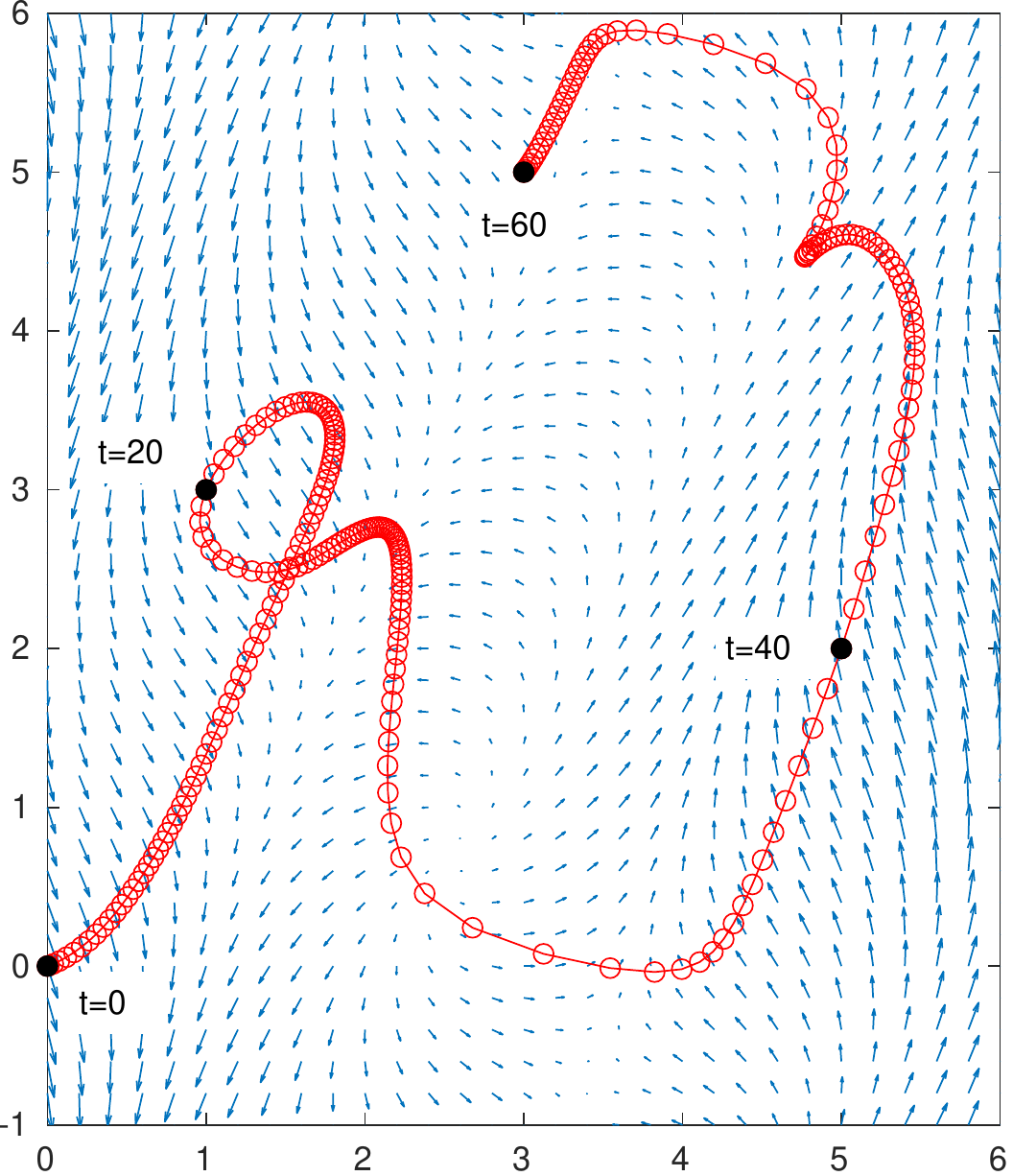}
	\end{center}
	\caption{An optimal trajectory for the second-order problem with interpolation nodes.}\label{fig:secondorder}
\end{figure}

\begin{remark} When implementing these methods, one can use some techniques to reduce their execution time and improve their behavior. One of them is starting with a coarse partition of the domain $[0,T]$ of the trajectory, that is, a low value of $N$. Once the discrete trajectory is reasonably stable, we refine the discretization by increasing the value of $N$, say 20\%, and interpolating. We continue iterating until it stabilizes again, and repeat until a desired value of $N$ is reached.

However, a low value of $N$ can make the trajectory unstable. In that case, we found that introducing a damping coefficient $\epsilon$ is useful. That is, if $\{q_k^{[k-1]}\}$ is the discrete trajectory at the current iteration and $\{\bar q_k^{[k-1]}\}$ is the adjusted trajectory computed by the method, we take the new discrete trajectory as $\{q_k^{[k-1]}+(1-\epsilon) (\bar q_k^{[k-1]}-q_k^{[k-1]})\}$.
\end{remark}

\section{Application to time-dependent systems}\label{examples-time-dependent}
\subsection{Zermelo's problem with varying wind}\label{sec:zermelotimedependent}
Here we revisit the time-optimal navigation problem of Section~\ref{section-zermelo}, considering a time-dependent vector field $W_t\in \mathfrak{X}(Q)$, with $t\in \mathbb{R}$.
Define the time-dependent Finsler metric $F_t(q,v_q)$ as in Section~\ref{section-zermelo}, using $W_t$ instead of $W$ in the expressions for $a$, $b$ and $\alpha$. Like its time-independent version, this is a positively 1-homogeneous function of $v_q$. The time it takes the ship to travel along a given curve
$\gamma\colon[s_0,s_N] \to Q$ is
\begin{equation*}
\int_{s_0}^{s_N} F_{t(s)}(\gamma(s), \dot \gamma(s))\,ds\,,
\end{equation*}
which is invariant under orientation-preserving reparametrizations of the curve. The function $t(s)$ relates the parameter $s$ and time, and satisfies the integral equation 
\begin{equation*}
t(s)=\int_{s_0}^s F_{t(\sigma)}\left(\gamma(\sigma),\dot\gamma(\sigma)\right)\, d\sigma\,,
 \end{equation*}
where we have set $t_0=t(s_0)=0$ for simplicity.

In order to discretize this computation, we write $h=(s_N-s_0)/N$, $s_k=s_0+kh$, $q_k=\gamma(s_k)$, $k=0,\dots,N$, and approximate $t(s_k)$ by
\begin{equation*}
  t(s_k)\approx \sum_{j=1}^{k} hF_{t_{j-1}}\left(q_{j-1},\frac{q_j-q_{j-1}}{h}\right)=:t_k, 
\end{equation*}
that is,
\begin{equation*}
 t_0=0,\qquad  t_k=t_{k-1} + hF_{t_{k-1}}\left(q_{k-1},\frac{q_k-q_{k-1}}{h}\right),\quad k=1,\dots,N.
\end{equation*}
Other choices such as the midpoint spatial discretization are also possible. This defines a sequence $\{t_k\}$ which must be updated as the Jacobi or Jacobi--Newton method proceeds. Note that this computation is sequential and cannot be performed in parallel. When implementing this, one could mitigate the impact on the performance by updating  $\{t_k\}$ every 100 or 1000 iterations for example.

Finally, define the time-dependent discrete Lagrangian as, for instance,
\[
L_{d,k}(q_k,q_{k+1})=\frac{h}{2} \left(F^2_{t_k}\left(q_k,\frac{q_{k+1}-q_k}{h}\right)+F^2_{t_{k+1}}\left(q_{k+1},\frac{q_{k+1}-q_k}{h}\right)\right),
\]
and use it to write the Jacobi or Jacobi--Newton methods. For example, the Jacobi method~\eqref{eq:DEL} becomes
\begin{equation*}
  D_2L_{d,k-1}(q_{k-1},\bar q_k)+D_1L_{d,k}(\bar q_k,q_{k+1})=0.
\end{equation*}

As an example, we applied this to the time-dependent vector field
\[W_t(x,y)=0.8 \sin(2x+y) \left(\cos(t/2),\sin(t/2))\right).\]
The ship is set to travel from $(1,6)$ to $(6,2)$ in minimum time. Here we used $N=50$. In Figure~\ref{fig:zermelotimedependent} the planned trajectory has red markers at $\gamma(s_k)$, $k=0,\dots,N$. The computed travel time $T$ was divided in $60$ equal intervals, which were used to plot the vectors in black, representing the wind that the ship will encounter when it passes through that point. The subfigures show the position of the ship and the surrounding vector field at times $0$, $T/3$, $2T/3$ and $T$. 

\begin{figure}[ht]
	\begin{center}
		\includegraphics[scale=.5]{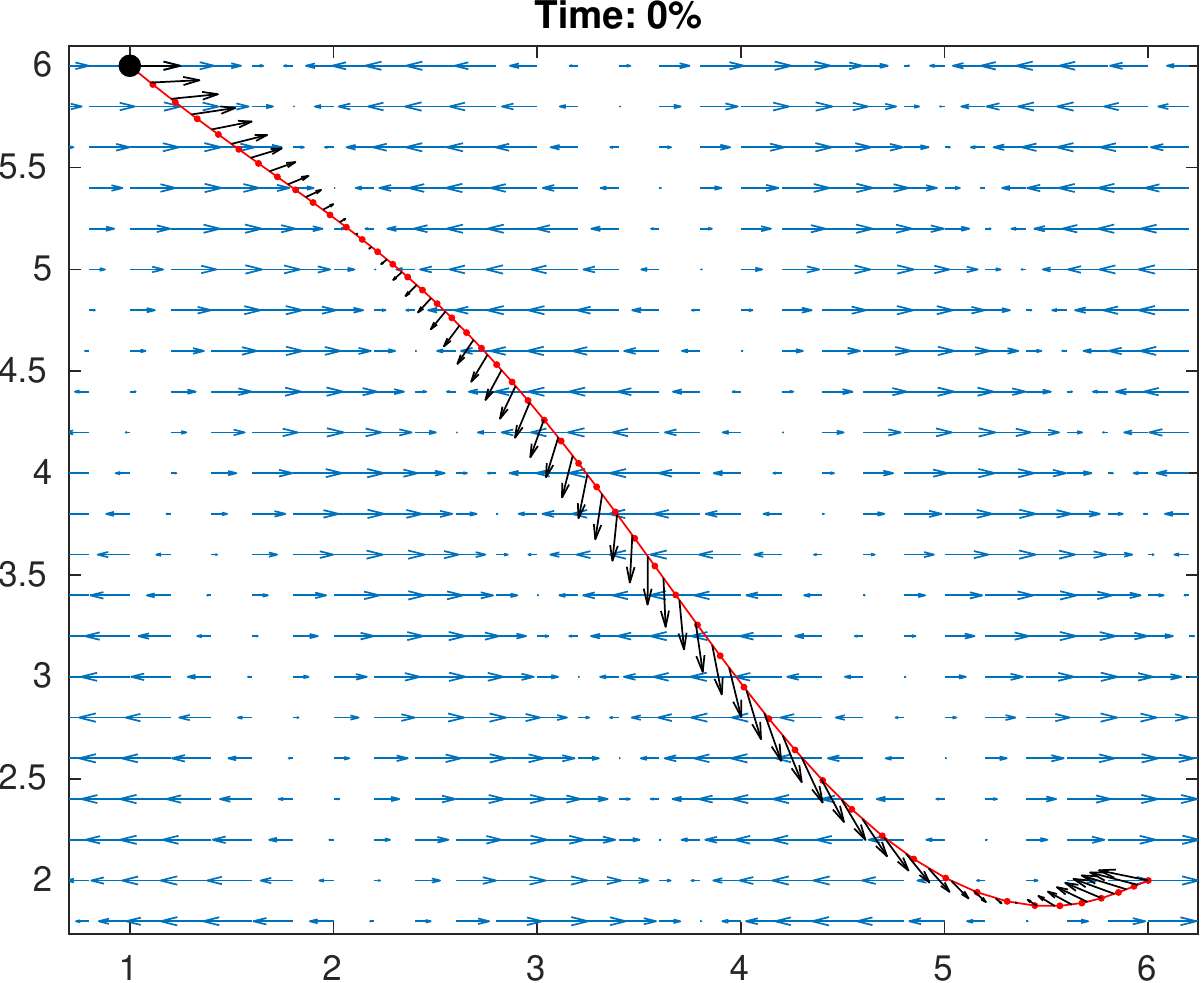}
		\includegraphics[scale=.5]{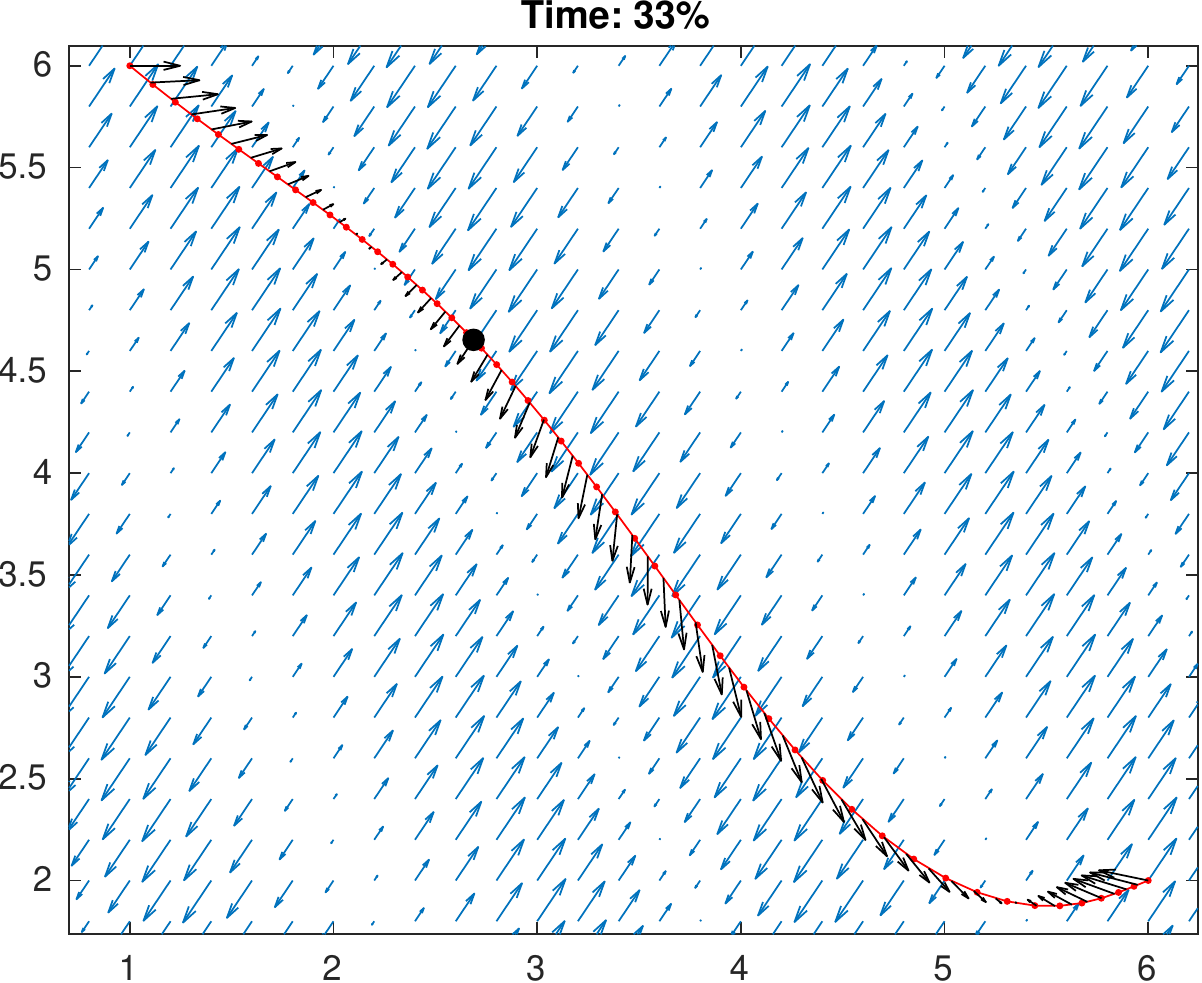}\\ \vspace{1ex}
		\includegraphics[scale=.5]{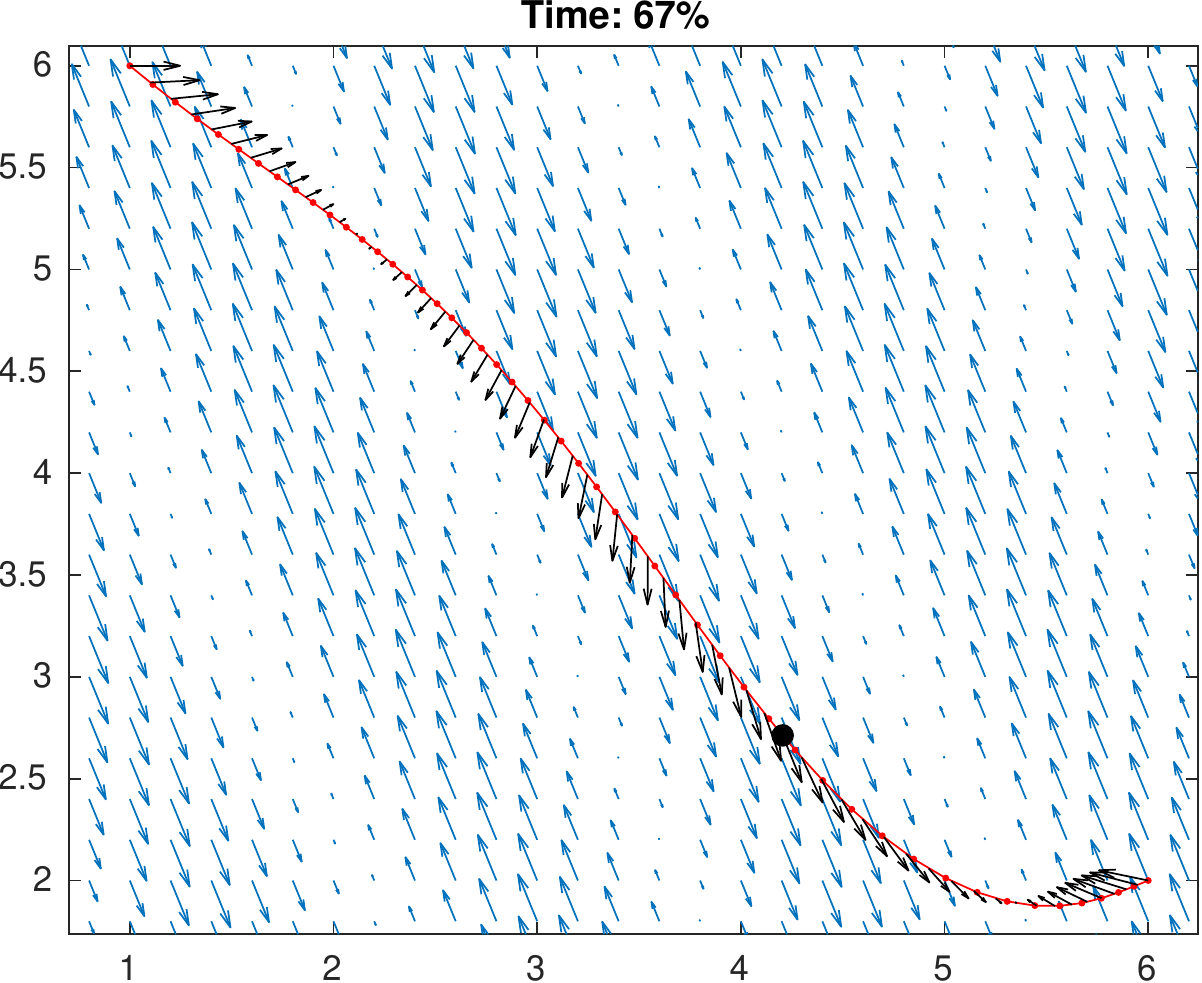}
		\includegraphics[scale=.5]{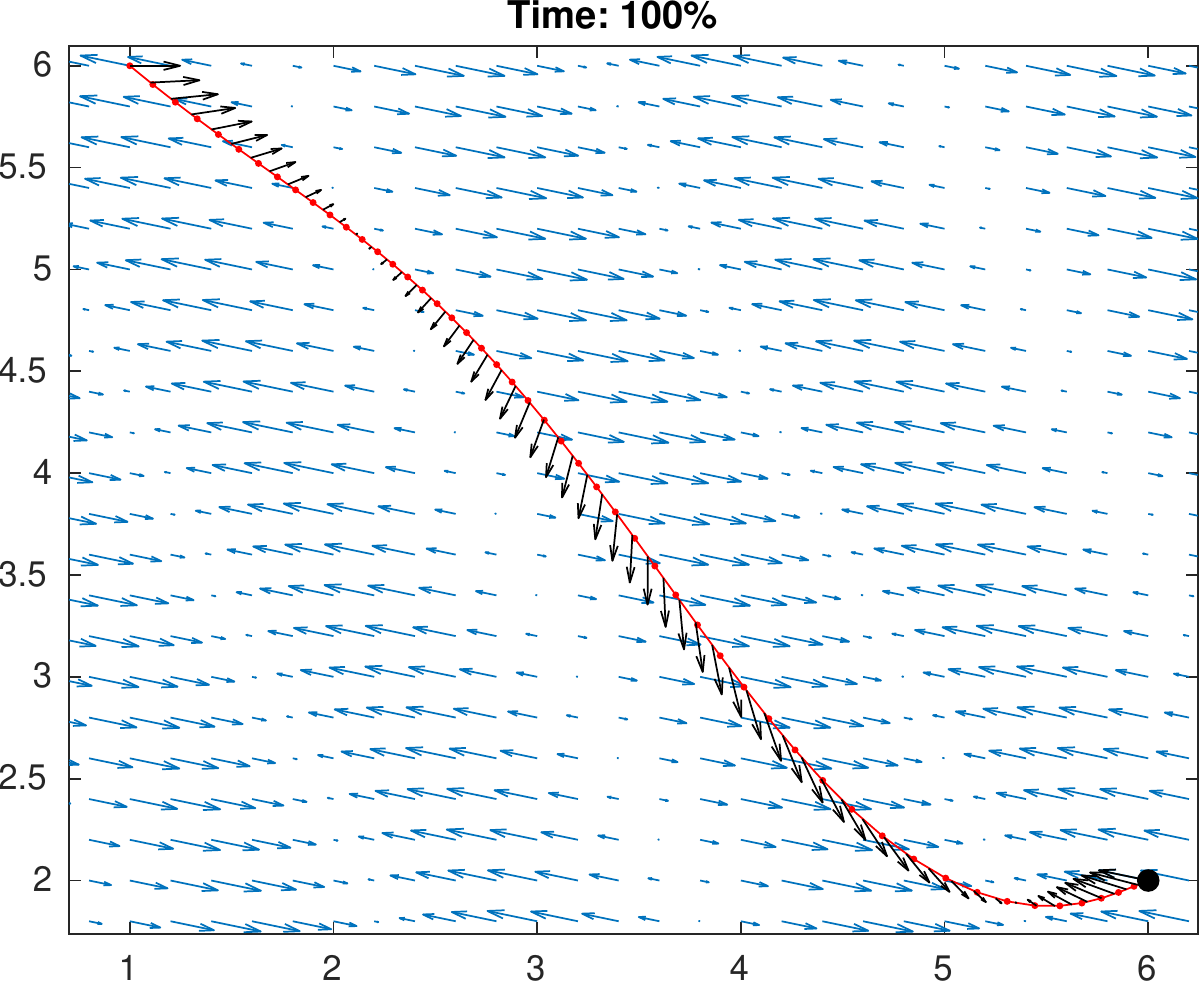}
	\end{center}
	\caption{An optimal trajectory with varying wind. The black dot is the ship, and the black vectors along the trajectory represent the wind that it will encounter in its journey. Each figure shows the vector field $W$ at different times. }\label{fig:zermelotimedependent}
\end{figure}

\subsection{Spacecraft trajectory planning}
The previous examples have been focused on variations of Zermelo's problem, but of course, our parallel iterative method can be applied in many other settings.
As a different example  consider a  fuel optimization problem for the controlled 4-body problem. As in Example~\ref{sec:zermelotimedependent}, it is necessary to use an extension of the parallel method for non-autonomous Lagrangian systems.

The model describes the dynamics of three bodies (Sun, Earth, Moon) among which a spacecraft is moving. Following a standard and simplified point of view \cite{Koon}, we  will assume that the three bodies move in a common plane, that the Moon rotates around the Earth in a circular motion and that Earth and Sun are both rotating in a circular motion about the center of mass of the three bodies. Observe  that in a Sun-Earth co-rotating frame both Sun and Earth are stationary. As usual, the mass of the spacecraft is assumed negligible. 
	The controlled equations  are
	\[
		\ddot{x}-2\dot{y}=\frac{\partial \Omega}{\partial x}+u_x\; ,\qquad 
			\ddot{y}-2\dot{x}=\frac{\partial \Omega}{\partial y}+u_y
			\]
		where 
		\[
		\Omega=\frac{(x+1)^2+y^2}{2}+\frac{m_S}{\sqrt{(x+1)^2+y^2}}+
			\frac{m_E}{\sqrt{x^2+y^2}} +\frac{m_M}{\sqrt{(x-x_M)^2+(y-y_M)^2}}
		\]
		where $m_S$, $m_E$ and $m_M$ are the normalized mass of Sun, Earth and Moon, respectively. The time-dependent  position of the Moon is $(x_M, y_M)$, with
		\begin{align*}
		\theta_M&=\omega_M t+\theta_{M0}\\
		x_M&= r_M\cos \theta_M\\
		y_M&= r_M \sin \theta_M
		\end{align*}
	 $\theta_{M0}$  is the initial angle of the Moon with respect to the $x$-axis in the Sun-Earth rotating frame, $r_M$ is the normalized radius of the Moon’s circular orbit, and $\omega_M$ is the normalized rotation rate of the Moon. Time is measured in years$/2\pi$.
	 
	As in Example \ref{fuel} the spacecraft seeks trajectories minimizing the fuel consumption functional given by
	\[
	\int_0^T (u_{x}^2+u_y^2)\, dt
	\]
	This can be written as a second order time-dependent Lagrangian function
	\[
	L(t, x, y, \dot{x}, \dot{y}, \ddot{x}, \ddot{y})=\left(\ddot{x}-2\dot{y}-\frac{\partial \Omega}{\partial x}\right)^2+	\left(\ddot{y}-2\dot{x}-\frac{\partial \Omega}{\partial y}\right)^2
	\]
  For our simulations we may use the same discretization as in Example~\ref{exp:trapezoidal} guaranteeing the convergence of our method.

In Figure~\ref{fig:moon-gravityassist} we show a trajectory that starts from a geosynchronous orbit, uses a ``gravity assist manoeuvre'' from the Moon and parks at the $\mathrm{L}_5$ point of the Earth-Moon system, in a prescribed total time of 8 days. We used $N=100$, and the initial guess consisted of two consecutive straight lines, deliberately passing through a point farther from Earth than the Moon to get a trajectory with these features. We remark that since the four bodies are treated as point masses, it is not unusual to get trajectories that go through the surface of the Moon. In order to avoid this, one could for example add a suitable penalty function to the Lagrangian.

\begin{figure}
		$\vcenter{\hbox{\includegraphics[scale=.65]{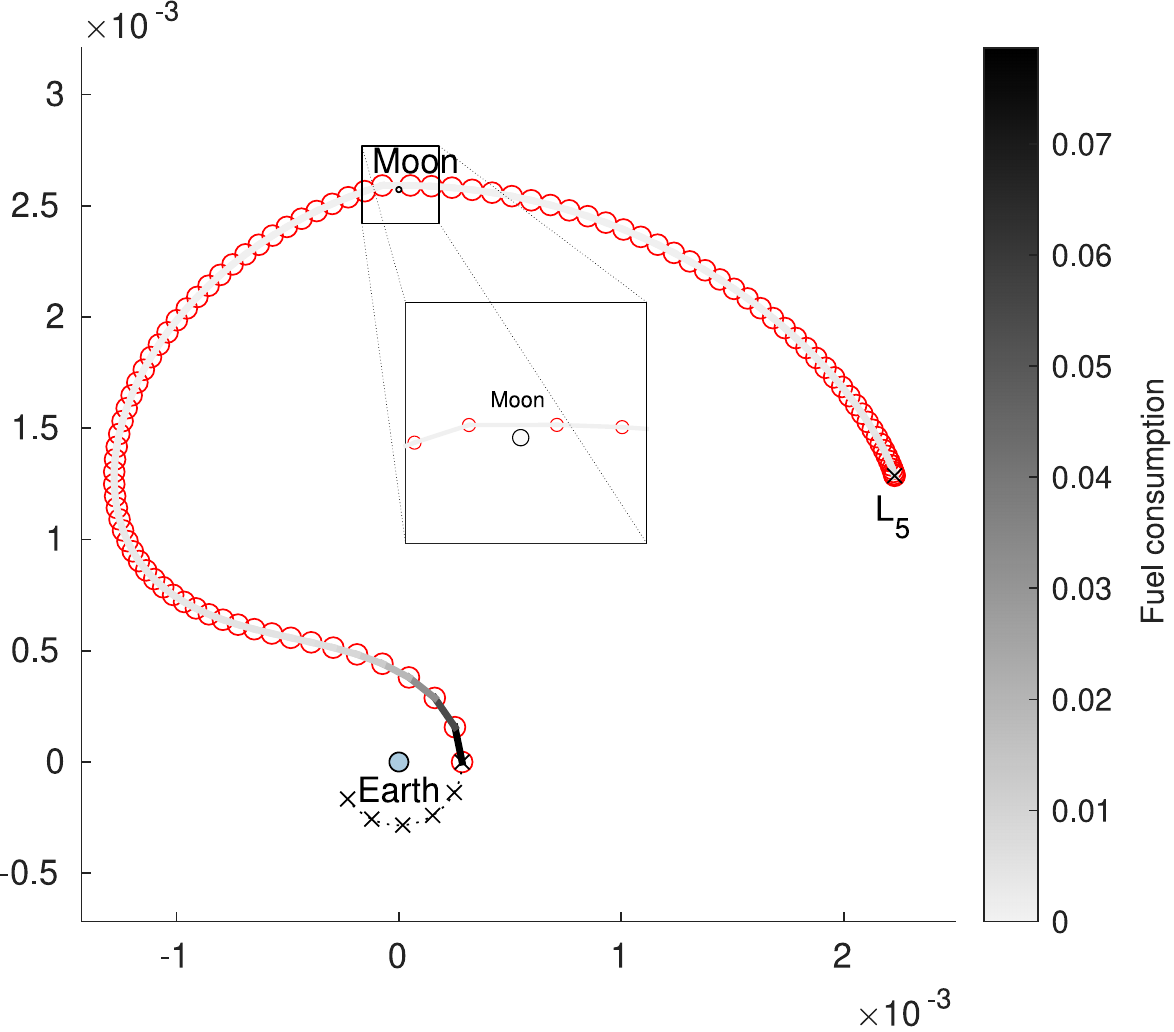}}}\hfill
		\vcenter{\hbox{\includegraphics[scale=.45]{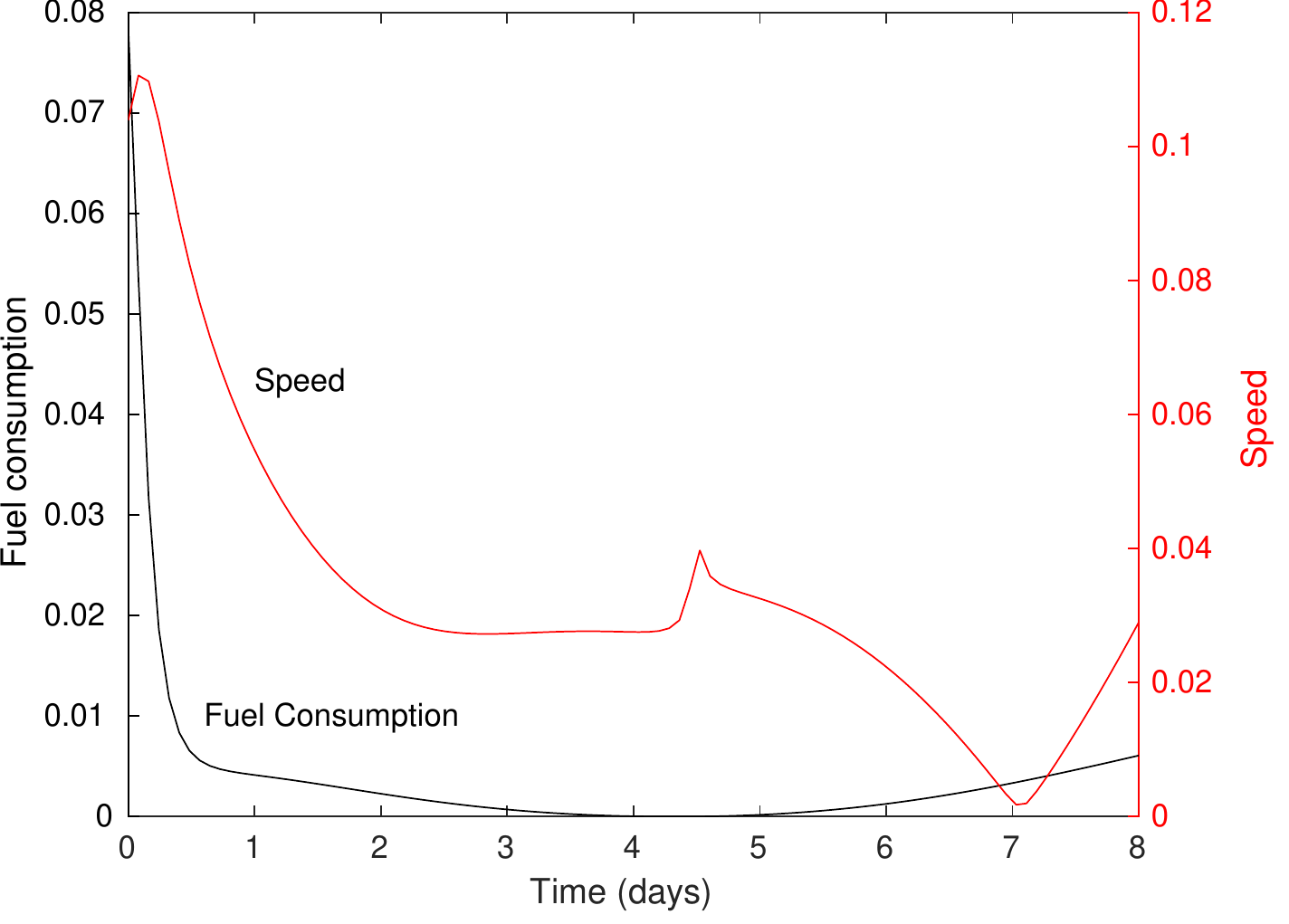}}}$
	\caption{Spacecraft trajectory starting from a geosynchronous orbit (in black, before starting the planned trajectory) and parking at the $\mathrm{L}_5$ Lagrange point of the Earth-Moon system (to scale). The trajectory is displayed in a rotating frame with the Moon fixed at the top of the diagram.}\label{fig:moon-gravityassist}
\end{figure}

\section{Conclusions and future work}
In this paper we have introduced an iterative numerical method admitting parallelization for discrete variational calculus proving the convergence of the method.  The applicability has been shown in some examples coming from navigation problems. 

Of course, these examples are only a small sample of application.  Our methods can be applied to problems in robotics and optimal control by incorporating real-time feedback, constraints, systems with external forces via the discrete Lagrange--d'Alembert principle \cite{MarsdenWest_ActaNum}, and trajectory correction, accounting for external perturbations or changes in the final endpoint conditions.  
We will study these generalizations in a future paper and, moreover, the extension of parallel methods adapted to invariant Lagrangian systems defined on a  Lie group. 

One variant we would like to discuss is the possibility of using an adaptive step size. If $q(t)$, $t\in[0,T]$, is a curve on $Q$, then we consider a time transformation $dt/d\tau=g(q^{[\gamma]}(t))$, where $g\colon T^{(\gamma)}Q \to \mathbb{R}$ is a smooth, positive function. This is called a \emph{Sundman transformation}  \cite[Ch.~9]{leimkuhler-reich}. In the discrete setting, each $q^{[\gamma-1]}_k$ is accompanied by a corresponding time $t_k$, satisfying
$\Delta t_k\approx\Delta \tau\, g_d(q^{[\gamma-1]}_{k-1},q^{[\gamma-1]}_k)$, along with $t_0=0$ and $t_N=T$. Here $g_d$ is a discretization of $g$ and $\Delta \tau$ is the fixed time step of the reparametrization of the solution curve (see also~\cite[Ch.~VIII.2]{hairer}). We propose that after each Jacobi step produces a sequence $\{q^{[\gamma-1]}_k\}$, we use these conditions, with approximation replaced by an equal sign, to compute all the $\Delta t_k$. This can be done in a very straightforward way. Then the times $t_k$ are updated to their new values.
We illustrate this in Figure~\ref{fig:moon-gravityassist-adaptive} for the adaptive time-step variation of the example used for Figure~\ref{fig:moon-gravityassist}. Note that the markers are $\tau$-equispaced. Comparing the two figures, the distribution of markers now emphasizes the detail of the trajectory near the Moon and Earth. Here we used a function $g$ depending on the position only.

\begin{figure}
		\includegraphics[scale=.6]{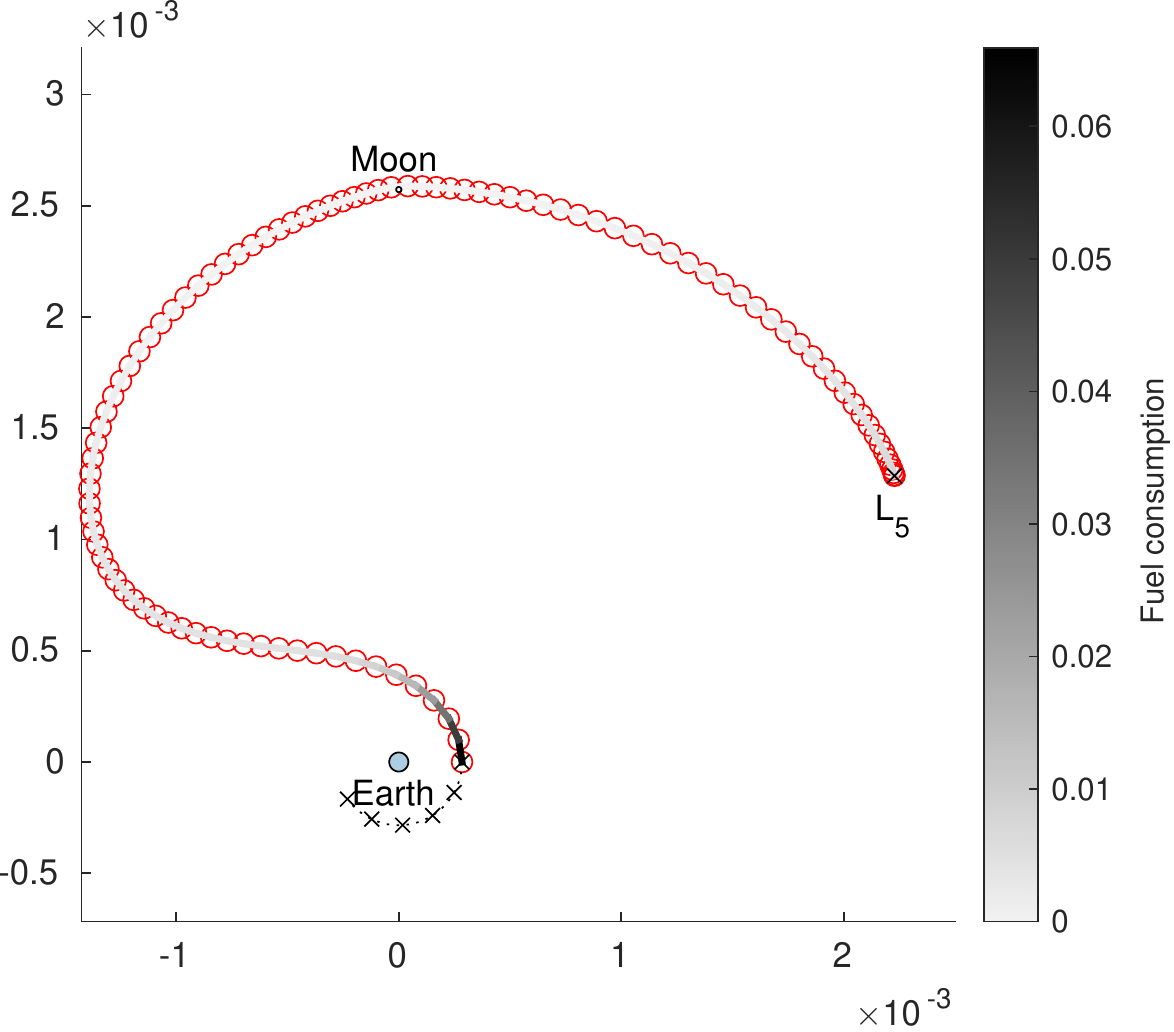}\hfill
		\includegraphics[scale=.6]{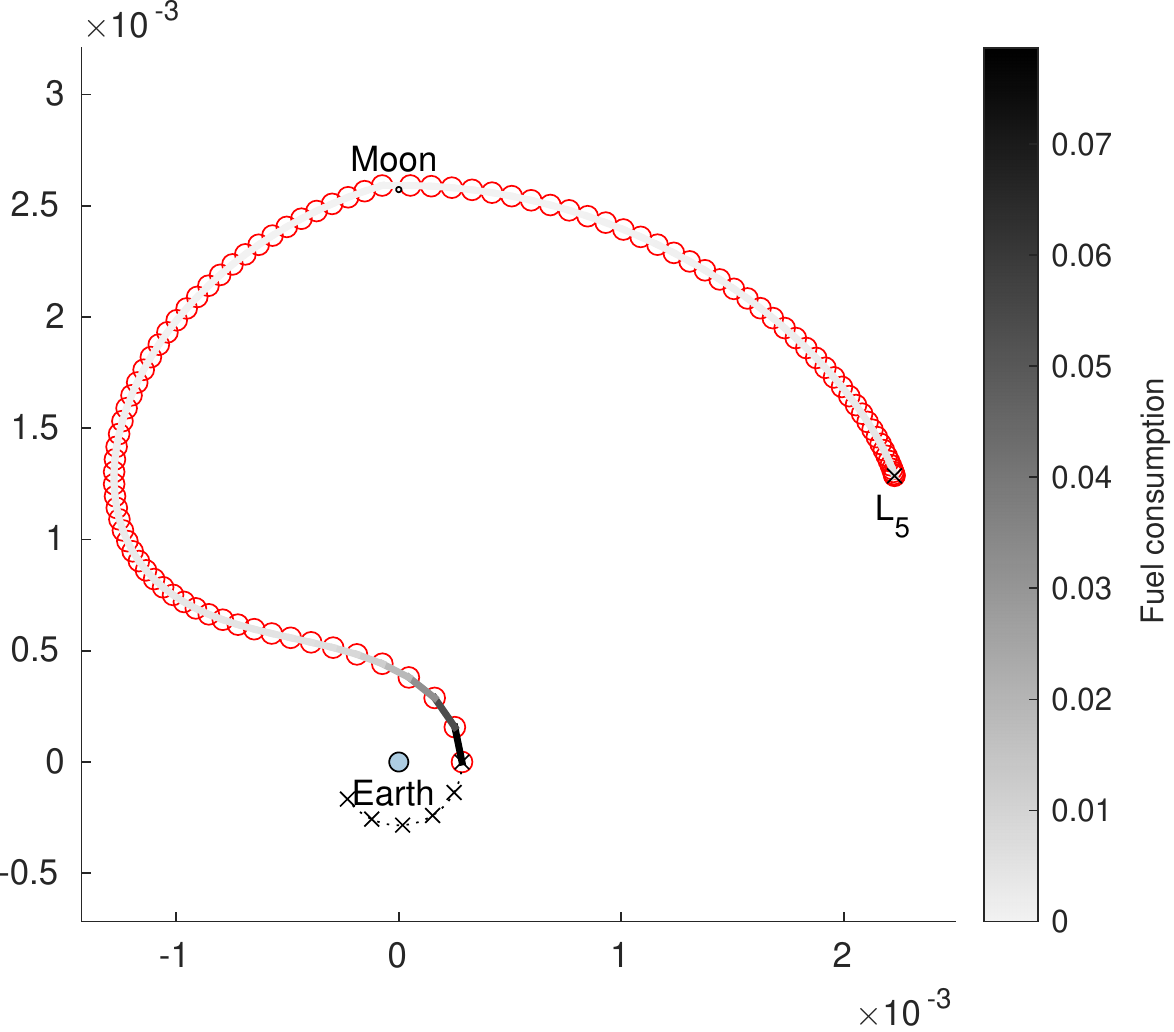}
	\caption{Left: adaptive step size. Right: fixed step size.}\label{fig:moon-gravityassist-adaptive}
\end{figure}


 \appendix
\section{Appendix}\label{appendixA}

Let us first define some matrices that we found useful for describing the structure of the Hessian matrices of the discrete $\gamma$-th order Lagrangians regarding powers of $h$. As such Lagrangians depend on $q^{[\gamma]}=(q^{(0)\,i},\dots,q^{(\gamma-1)\,i})$, we are indexing the matrix rows and columns from $0$ to $\gamma-1$.

\begin{itemize}
\item The upper-triangular matrix $\mathbf{A}(h) \in M_{\gamma}(\mathbb{R})$, with non-zero entries
\begin{equation*}
a_{\alpha \beta}(h) = \frac{h^{\beta - \alpha}}{(\beta - \alpha)!}\,, \quad \text{for } 0 \leq \alpha \leq \beta \leq \gamma-1\,.
\end{equation*}
\item The symmetric matrices $\mathbf{B}(h), \mathbf{C}(h) \in M_{\gamma}(\mathbb{R})$, with entries
\begin{alignat*}{2}
b_{\alpha \beta}(h) &= \frac{h^{2 \gamma - \alpha - \beta - 1}}{(2 \gamma - \alpha - \beta - 1)!}\,, \quad && \text{for } 0 \leq \alpha, \beta \leq \gamma-1\,;\\
c_{\alpha \beta}(h) &= \frac{h^{2 \gamma - \alpha - \beta - 1}}{(2 \gamma - \alpha - \beta - 1) (\gamma - \alpha - 1)! (\gamma - \beta - 1)!}\,. \quad && 
\end{alignat*}
\item The diagonal matrix $\mathbf{D} \in M_{\gamma}(\mathbb{R})$, with entries
\begin{equation*}
d_{\alpha\, \beta} = (-1)^{\gamma - \alpha - 1}\delta_{\alpha,\, \beta}\,,\quad \text{for } 0 \leq \alpha, \beta \leq \gamma-1\,.
\end{equation*}
\item The exchange matrix $\mathbf{E} \in M_{\gamma}(\mathbb{R})$, with entries
\begin{equation*}
e_{\alpha\, \beta} = \delta_{\gamma - \alpha - 1,\beta}\,,\quad \text{for } 0 \leq \alpha, \beta \leq \gamma-1\,.
\end{equation*}
\end{itemize}

For example, for $\gamma=3$ these matrices are
\begingroup
\renewcommand*{\arraystretch}{1.4}
\begin{align*}
\mathbf{A}(h)&=\begin{pmatrix}
  1&h&\frac{h^2}{2}\\
  0&1&h\\
  0&0&1
\end{pmatrix},&
\mathbf{B}(h)&=\begin{pmatrix}
  \frac{h^5}{120}&\frac{h^4}{24}&\frac{h^3}{6}\\
  \frac{h^4}{24}&\frac{h^3}{6}&\frac{h^2}{2}\\
  \frac{h^3}{6}&\frac{h^2}{2}&h
\end{pmatrix},&
\mathbf{C}(h)&=\begin{pmatrix}
  \frac{h^5}{20}&\frac{h^4}{8}&\frac{h^3}{6}\\
  \frac{h^4}{8}&\frac{h^3}{3}&\frac{h^2}{2}\\
  \frac{h^3}{6}&\frac{h^2}{2}&h
\end{pmatrix},\\
\mathbf{D}&=\begin{pmatrix}
  1&0&0\\ 0&-1&0\\ 0&0&1
\end{pmatrix},&
\mathbf{E}&=\begin{pmatrix}
  0&0&1\\ 0&1&0\\ 1&0&0
\end{pmatrix}.&
\end{align*}
\endgroup

\begin{lemma}
\label{cor:AB_matrix} The following identities hold:
\begin{enumerate}
\item $\mathbf{A}(-h) = \mathbf{D} \mathbf{A}(h) \mathbf{D} = \mathbf{A}^{-1}(h)$,
\item $\mathbf{B}(-h) = - \mathbf{D} \mathbf{B}(h) \mathbf{D}$.
\end{enumerate}
\end{lemma}

\begin{proof}
\begin{itemize}
\item It is not difficult to see that
\begin{equation*}
\sum_{\rho = 0}^{\gamma-1} \sum_{\epsilon = 0}^{\gamma-1} d_{\alpha \rho} a_{\rho \epsilon}(h) d_{\epsilon \beta} = (-1)^{2 \gamma - \alpha - \beta - 2} \frac{h^{\beta - \alpha}}{(\beta - \alpha)!} = \frac{(-h)^{\beta - \alpha}}{(\beta - \alpha)!} = a_{\alpha \beta}(-h)\,.
\end{equation*}
In order to show the remaining identity one only needs to check that due to upper-triangular nature of $\mathbf{A}(h)$,
\begin{equation*}
\sum_{\epsilon = 0}^{\gamma-1} a_{\alpha \epsilon}(-h) a_{\epsilon \beta}(h) = \sum_{\epsilon = \alpha}^{\beta} a_{\alpha \epsilon}(-h) a_{\epsilon \beta}(h) = h^{\beta - \alpha} \sum_{\epsilon = \alpha}^{\beta} \frac{(-1)^{\epsilon - \alpha}}{(\epsilon - \alpha)!(\beta - \epsilon)!}\,.
\end{equation*}
To simplify this we relabel $k = \epsilon - \alpha$, $m = \beta - \alpha$ and apply the binomial theorem, finally obtaining
\begin{equation*}
\sum_{k = 0}^{m} \frac{(-1)^{k}}{k! \, (m - k)!} = \frac{1}{m!} \sum_{k = 0}^{m} \binom{m}{k} (-1)^k = \begin{cases}
  1&\text{if }m=0,\\
  0^m&\text{otherwise},
\end{cases}
\end{equation*}
which proves our claim.
\item Direct multiplication leads to the desired result
\begin{align*}
\sum_{\epsilon = 0}^{\gamma - 1} \sum_{\lambda = 0}^{\gamma - 1} d_{\alpha \epsilon} b_{\epsilon \lambda}(h) d_{\lambda \beta} &= (-1)^{\gamma - \alpha - 1} (-1)^{\gamma - \lambda - 1} \frac{h^{2 \gamma - \epsilon - \lambda - 1}}{(2 \gamma - \epsilon - \lambda - 1)!} \delta_{\alpha,\, \epsilon} \delta_{\lambda,\, \beta}\\
&= (-1)^{2 \gamma - \alpha - \beta - 2} \frac{h^{2 \gamma - \alpha - \beta - 1}}{(2 \gamma - \alpha - \beta - 1)!} = - b(-h)\,.\qedhere
\end{align*}
\end{itemize}
\end{proof}

\begin{lemma}
\label{lem:C_matrix}
The matrix $\mathbf{C}(h)$
\begin{enumerate}
\item satisfies $\mathbf{C}(h) = \mathbf{A}(h) \mathbf{D} \mathbf{B}(h)$;
\item \label{lem:C_matrix.itm:C_symm} satisfies $\mathbf{C}(h) = -\mathbf{D} \mathbf{C}(-h) \mathbf{D}$;
\item \label{lem:C_matrix.itm:decomp} admits a decomposition $\mathbf{C}(h) = \mathbf{L}(h)\mathbf{U}(h)$, where $\mathbf{L}(h), \mathbf{U}(h) \in M_{\gamma}(\mathbb{R})$ are respectively lower and upper triangular matrices with entries
\begin{alignat*}{2}
l_{\alpha \beta} &= \frac{h^{\beta - \alpha} \alpha! \, (2 \gamma - \beta - 1)! \, (\gamma - \beta - 1)! \, (2 \gamma - \alpha - \beta - 2)!}{\beta!\, (\alpha - \beta)!\, (2 \gamma - \alpha - 1)! \, (\gamma - \alpha - 1)! \, (2 \gamma - 2 \beta - 2)!}, \quad &\text{for } 0 \leq \beta \leq \alpha \leq \gamma-1\,;\\
u_{\alpha \beta} &= \frac{h^{2 \gamma - \alpha - \beta - 1} \, \alpha! \, \beta! \, (2 \gamma - 2 \alpha - 1)! \,(2 \gamma - \alpha - \beta - 2)!}{(\beta - \alpha)! \, (2 \gamma - \alpha - 1)! \, (2 \gamma - \beta - 1)! \, (\gamma - \alpha - 1)! \, (\gamma - \beta - 1)!}, \quad &\text{for } 0 \leq \alpha \leq \beta \leq \gamma-1\,;
\end{alignat*}
\item is regular and positive-definite, with determinant
\begin{equation*}
\det \mathbf{C}(h) = h^{\gamma^2} \prod_{\alpha = 0}^{\gamma - 1} \frac{\alpha!}{(\gamma + \alpha)!}\,.
\end{equation*}
\end{enumerate}
\end{lemma}

Before proving this lemma, we write down the matrices $\mathbf{L}(h)$ and $\mathbf{U}(h)$ for the case $\gamma=3$ as above:
\begingroup
\renewcommand*{\arraystretch}{1.4}
\begin{equation*}
\mathbf{L}(h)=\begin{pmatrix}
  1&0&0\\
  \frac{5}{2h}&1&0\\
  \frac{10}{3h^2}&\frac{4}{h}&1
\end{pmatrix},\qquad
\mathbf{U}(h)=\begin{pmatrix}
  \frac{h^5}{20}&\frac{h^4}{8}&\frac{h^3}{6}\\
  0&\frac{h^3}{48}&\frac{h^2}{12}\\
  0&0&\frac{h}{9}
\end{pmatrix}.
\end{equation*}
\endgroup

\begin{proof}
\begin{enumerate}
\item Performing the multiplication explicitly we obtain
\begin{equation*}
\sum_{\epsilon = \alpha}^{\gamma-1} \sum_{\sigma = 0}^{\gamma - 1} a_{\alpha \epsilon}(h) d_{\epsilon \sigma} b_{\sigma \beta}(h) = h^{2\gamma-\alpha-\beta-1} \sum_{\epsilon = \alpha}^{\gamma - 1} \frac{(-1)^{\gamma - \epsilon - 1}}{(\epsilon - \alpha)!(2 \gamma - \epsilon - \beta - 1)!}\,.
\end{equation*}
Setting $a = \gamma - \alpha - 1$, $b = 2 \gamma - \alpha - \beta - 1$ ($b > a$) and $c = \epsilon - \alpha$, we may transform this expression into
\begin{equation}
\label{eq:C_product_identity}
(-1)^{\gamma-\alpha-1} \frac{h^{2 \gamma - \alpha - \beta - 1}}{(2\gamma - \alpha - \beta - 1)!} \sum_{c = 0}^{a} \binom{b}{c} (-1)^c\,.
\end{equation}
Now, using Pascal's identity we obtain the telescopic sum
\begin{eqnarray*}
  \sum_{c = 0}^{a} \binom{b}{c} (-1)^c&=&\sum_{c = 0}^{a} \left( \binom{b-1}{c-1}+\binom{b-1}{c}\right) (-1)^c\\
  &=&\binom{b-1}{-1}+\binom{b - 1}{a} (-1)^{a}=\binom{b - 1}{a} (-1)^{a}\,.
\end{eqnarray*}
Here we have used $\binom{b-1}{-1}=0$, a convention that can be justified by adding zeros outside the rows of Pascal's triangle.

Substituting this in~\eqref{eq:C_product_identity} and writing everything in terms of $\alpha$, $\beta$ and $\gamma$, we obtain
\begin{equation*}
 \frac{h^{2 \gamma - \alpha - \beta - 1}}{(2\gamma - \alpha - \beta - 1)!} \binom{2 \gamma - \alpha - \beta - 2}{ \gamma - \alpha - 1}=c_{\alpha\beta}(h).
\end{equation*}

\item Using Lemma~\ref{cor:AB_matrix} together with the former identity and the fact that $\mathbf{D}^{-1} = \mathbf{D}$ leads us to
\begin{equation*}
\mathbf{D} \mathbf{C}(h) \mathbf{D} = \mathbf{D} \mathbf{A}(h) \mathbf{D} \mathbf{B}(h) \mathbf{D} = \mathbf{A}(-h) \mathbf{B}(h) \mathbf{D} = - \mathbf{A}(-h) \mathbf{D} \mathbf{B}(-h) = -\mathbf{C}(-h)\,.
\end{equation*}

\item To prove this, it suffices to show that
\begin{equation*}
\sum_{\delta = 0}^{\gamma} l_{\alpha \delta}(h) \, u_{\delta \beta}(h) = \sum_{\delta = 0}^{\min(\alpha,\beta)} l_{\alpha \delta}(h) \, u_{\delta \beta}(h) = c_{\alpha \beta}(h)\,,
\end{equation*}
where we have made use of the triangular structure of $\mathbf{L}(h)$ and $\mathbf{U}(h)$. Substituting the corresponding expressions we get
\begin{align*}
\sum_{\delta = 0}^{\min(\alpha,\beta)} l_{\alpha \delta}(h) \, u_{\delta \beta}(h) &= \frac{h^{2 \gamma - \alpha - \beta - 1} \alpha! \, \beta!}{(2\gamma - \alpha - 1)! (2\gamma - \beta - 1)! (\gamma - \alpha - 1)! (\gamma - \beta - 1)!}\\
&\times \sum_{\delta = 0}^{\min(\alpha,\beta)} \frac{(2 \gamma - 2\delta - 1) (2 \gamma - \alpha - \delta - 2)! (2 \gamma - \beta - \delta - 2)!}{(\alpha - \delta)!(\beta - \delta)!}\,.
\end{align*}
For this to be equal to $c_{\alpha\beta}(h)$, we need to show that
\begin{equation*}
\sum_{\delta = 0}^{\min(\alpha,\beta)} \frac{(2 \gamma - 2\delta - 1) (2 \gamma - \alpha - \delta - 2)! (2 \gamma - \beta - \delta - 2)!}{(\alpha - \delta)!(\beta - \delta)!} = \frac{(2 \gamma - \alpha - 1)! (2 \gamma - \beta - 1)!}{(2 \gamma - \alpha - \beta - 1) \alpha! \, \beta!}\,.
\end{equation*}
In order to prove this, we can apply Zeilberger's algorithm, also known as \emph{creative telescoping} \cite{Zeilberger}. First, notice that the expression is symmetric in $\alpha$ and $\beta$, so we may choose $\min(\alpha,\beta) = \alpha$ without loss of generality. Denoting the summand of the former equation as $F(\alpha,\delta)$, the algorithm gives us
\begin{equation*}
G(\alpha,\delta) = -\frac{(2 \gamma - \alpha - \delta - 1)! (2 \gamma - \beta - \delta - 1)!}{(2 \gamma - \alpha - \beta - 1) (\alpha - \delta)! (\beta - \delta)!}
\end{equation*}
which satisfies that $F(\alpha, \delta) = G(\alpha, \delta + 1) - G(\alpha, \delta)$. This transforms our sum into a telescopic series,
\begin{align*}
\sum_{\delta = 0}^{\alpha} F(\alpha, \delta) &= \sum_{\delta = 0}^{\alpha} [ G(\alpha, \delta + 1) - G(\alpha, \delta) ]\\
&= G(\alpha, \alpha + 1) - G(\alpha, 0) = 0 + \frac{(2 \gamma - \alpha - 1)! (2 \gamma - \beta - 1)!}{(2 \gamma - \alpha - \beta - 1) \alpha! \, \beta!}\,.
\end{align*}
When evaluating $G(\alpha, \alpha + 1)$ we have used the convention $(-1)!=\infty$, which is customary in combinatorics. Our decomposition is thus proved.
\item From the definition of $\mathbf{L}(h)$ in item \ref{lem:C_matrix.itm:decomp}, it is easy to check that $l_{\alpha \alpha}(h) = 1$ for $\alpha = 0,\dots,\gamma-1$. The decomposition is therefore the standard $LU$ decomposition of the matrix $\mathbf{C}(h)$. Thus, all the information that we need to determine regularity, definiteness and the determinant itself is encoded in $\mathbf{U}(h)$.

It is also immediate to check that for $h \neq 0$
\begin{equation*}
u_{\alpha \alpha}(h) = \frac{h^{2 \gamma - 2\alpha - 1} \, \alpha!^2 \,(2 \gamma - 2\alpha - 2)! \,(2 \gamma - 2\alpha - 1)! }{(\gamma - \alpha - 1)!^2 \, (2 \gamma - \alpha - 1)!^2 } \neq 0\,
\end{equation*}
which proves regularity of $\mathbf{C}(h)$. Moreover, we may further decompose $\mathbf{U}(h)$ uniquely as $\widehat{\mathbf{D}}(h) \widehat{\mathbf{U}}(h)$, where the former is the diagonal matrix formed by the diagonal entries of $\mathbf{U}(h)$ and the latter has ones on its diagonal. Due to the symmetry of $\mathbf{C}(h)$, one has that $(\mathbf{L}(h) \widehat{\mathbf{D}}(h) \widehat{\mathbf{U}}(h))^{\top} = \widehat{\mathbf{U}}^{\top}(h) \widehat{\mathbf{D}}(h) \mathbf{L}^{\top}(h) = \mathbf{L}(h) \widehat{\mathbf{D}}(h) \widehat{\mathbf{U}}(h)$. This gives two $LU$ decompositions of $\mathbf{C}(h)$ with the lower triangular factors having ones on their diagonals, and from the uniqueness of such a decomposition we obtain that $\widehat{\mathbf{U}}(h) = \mathbf{L}^{\top}(h)$. Thus $\mathbf{C}(h) = \mathbf{L}(h) \widehat{\mathbf{D}}(h) \mathbf{L}^{\top}(h)$ and by Sylvester's law of inertia, $\mathbf{C}(h)$ must be positive-definite for all $h > 0$ and negative definite for all $h < 0$. Finally, we can compute the determinant of $\mathbf{C}(h)$ by computing the product
\begin{align*}
\det \mathbf{C}(h)&=\det \mathbf{U}(h)=\prod_{\alpha = 0}^{\gamma - 1} u_{\alpha \alpha}(h) = \prod_{\alpha = 0}^{\gamma - 1} \frac{h^{2 \gamma - 2\alpha - 1} \, \alpha!^2 \,(2 \gamma - 2\alpha - 2)! \,(2 \gamma - 2\alpha - 1)! }{(\gamma - \alpha - 1)!^2 \, (2 \gamma - \alpha - 1)!^2 }\\
&= h^{\sum_{\alpha = 0}^{\gamma - 1} \left[ 2 \gamma - 2 \alpha - 1\right]} \prod_{\alpha = 0}^{\gamma - 1} \frac{ (2 \gamma - 2\alpha - 2)! \,(2 \gamma - 2\alpha - 1)!}{(2 \gamma - \alpha - 1)!^2}\\
&= h^{\gamma^2} \prod_{\lambda = 0}^{\gamma - 1} \frac{ (2 \lambda + 1)! \,(2 \lambda)!}{(\gamma + \lambda)!^2} = h^{\gamma^2} \prod_{\lambda = 0}^{\gamma - 1} \frac{\lambda!}{(\gamma + \lambda)!}
\end{align*}
On the second line we have made use of the fact that $\prod_{\alpha = 0}^{\gamma - 1} \alpha! = \prod_{\alpha = 0}^{\gamma - 1} (\gamma - \alpha - 1)!$. On the third we  have first used the arithmetic series formula and performed the index relabelling $\lambda = \gamma - \alpha - 1$, and finally we have used that $\prod_{\lambda = 0}^{\gamma - 1} (2 \lambda + 1)! (2 \lambda)! = \prod_{\lambda = 0}^{2 \gamma - 1} \lambda !$ and $\prod_{\lambda = 0}^{\gamma - 1} (\gamma + \lambda)! = \prod_{\lambda = \gamma}^{2 \gamma - 1} \lambda !$.\qedhere
\end{enumerate}
\end{proof}

\begin{corollary}
\label{cor:B_matrix}
The matrix $\mathbf{B}(h)$ is regular with determinant
\begin{equation*}
\det \mathbf{B}(h) = (-1)^{\gamma (\gamma - 1)/2} \det \mathbf{C}(h)
\end{equation*}
\end{corollary}

\begin{proof}
Since $\mathbf{C}(h) = \mathbf{A}(h) \mathbf{D} \mathbf{B}(h)$ and $\mathbf{A}(h)$, $\mathbf{D}$ and $\mathbf{C}(h)$ are regular, it is clear that $\mathbf{B}(h)$ must also be regular. Moreover, since $\det \mathbf{A}(h) = 1$ and
\begin{equation*}
\det \mathbf{D} = \prod_{\alpha = 0}^{\gamma - 1} (-1)^{\alpha} = (-1)^{\sum_{\alpha = 0}^{\gamma - 1} \alpha} = (-1)^{\gamma (\gamma - 1)/2},
\end{equation*}
where we have once more used the arithmetic series formula, we get the desired result.
\end{proof}


\section*{Acknowledgments}
D. Mart{\'\i}n de Diego acknowledges financial support from the Spanish Ministry of Science and Innovation, under grant PID2019-106715GB-C21 and  the ``Severo Ochoa Programme for Centres of Excellence'' in R\&D (CEX2019-000904-S).
S.\ Ferraro acknowledges financial support from PICT 2019-00196, FONCyT, Argentina, and PGI 2018, UNS.



\newcommand\oneletter[1]{#1}\newcommand\Yu{Yu}

\end{document}